\documentclass [11pt,reqno]{amsart}

\usepackage{amsmath,amssymb,verbatim,geometry,color}
\usepackage[all]{xy}
\usepackage{mathrsfs}
\usepackage[backref,pagebackref,pdftex,hyperindex]{hyperref}
\usepackage{graphicx}

\usepackage{tikz}
\usepackage{tikz-cd}

\def\XXint#1#2#3{{\setbox0=\hbox{$#1{#2#3}{\int}$ }
\vcenter{\hbox{$#2#3$ }}\kern-.6\wd0}}

\newcommand{\B}{\mathbb{B}}
\newcommand{\C}{\mathbb{C}}
\newcommand{\G}{\mathbb{G}}

\newcommand{\N}{\mathbb{N}}
\renewcommand{\P}{\mathbb{P}}
 \newcommand{\Q}{\mathbb{Q}}
 \newcommand{\R}{\mathbb{R}}
 \newcommand{\Z}{\mathbb{Z}}
\newcommand{\T}{\mathbb{T}}

\newcommand{\fa}{\mathfrak{a}}
\newcommand{\fb}{\mathfrak{b}}

\newcommand{\cA}{\mathcal{A}}

\newcommand{\cE}{\mathcal{E}}
\newcommand{\cF}{\mathcal{F}}

\newcommand{\cH}{\mathcal{H}}
\newcommand{\cJ}{\mathcal{J}}
\newcommand{\cL}{\mathcal{L}}

\newcommand{\cO}{\mathcal{O}}

\newcommand{\cZ}{\mathcal{Z}}
\newcommand{\cX}{\mathcal{X}}
\newcommand{\cY}{\mathcal{Y}}

\newcommand{\tcX}{\widetilde{\mathcal{X}}}

\newcommand{\om}{\omega}

\newcommand{\p}{\psi}

\newcommand{\ie}{{\rm i.e.\ }}

\newcommand{\an}{\mathrm{an}}

\DeclareMathOperator{\mab}{M}

\DeclareMathOperator{\hh}{H}

\DeclareMathOperator{\jj}{J}
\DeclareMathOperator{\env}{P}

\DeclareMathOperator{\Aut}{Aut}

\DeclareMathOperator{\codim}{codim}

\DeclareMathOperator{\DF}{DF}

\DeclareMathOperator{\Exc}{Exc}

\DeclareMathOperator{\Hom}{Hom}

\DeclareMathOperator{\MA}{MA}

\DeclareMathOperator{\Spec}{Spec}

\DeclareMathOperator{\vol}{vol}
\DeclareMathOperator{\Vol}{Vol}

\DeclareMathOperator{\ord}{ord}

\DeclareMathOperator{\PSH}{PSH}

\DeclareMathOperator{\tr}{tr}
\DeclareMathOperator{\redu}{red}

\DeclareMathOperator{\Ric}{Ric}

\newcommand{\ddc}{dd^c}
\newcommand{\dc}{d^c}
\newcommand{\de}{d}

\renewcommand{\div}{\mathrm{div}}

\newcommand{\NA}{\mathrm{NA}}

\newcommand{\D}{\Delta}

\newcommand{\cro}[1]{[\![#1]\!]}
\newcommand{\lau}[1]{(\!(#1)\!)}

\numberwithin{equation}{section}       
\newtheorem{thm}{Theorem}[section]
\newtheorem{prop}[thm] {Proposition}
\newtheorem{defi}[thm] {Definition}
\newtheorem{lem}[thm] {Lemma}

\newtheorem{cor}[thm]{Corollary}
\newtheorem{prop-def}[thm]{Proposition-Definition}

\newtheorem{rmk}[thm]{Remark}

\newtheorem{conj}[thm]{Conjecture}

\newtheorem{mainthm}{Theorem}

\newtheorem{maincor}{Corollary}

\makeatletter
\newtheorem*{rep@theo}{\rep@title}
\newcommand{\newreptheo}[2]{%
\newenvironment{rep#1}[1]{%
 \def\rep@title{#2 \ref{##1}}%
 \begin{rep@theo}}%
 {\end{rep@theo}}}
\makeatother

\newreptheo{theo}{Theorem}

\newreptheo{conjecture}{Conjecture}

\newreptheo{corol}{Corollary}

\newreptheo{Defin}{Definition}

 \theoremstyle{plain}
\newtheorem*{namedthm}{\namedthmname}
\newcounter{namedthm}

\makeatletter

\theoremstyle{remark}

\hypersetup{
    unicode=false,        
    pdftoolbar=true,      
    pdfmenubar=true,       
    pdffitwindow=false,     
    pdfstartview={FitH},    
    pdftitle={A solution to the Yau-Tian-Donaldson conjecture through special Fujita Approximations},    
    pdfauthor={Trusiani},     
    colorlinks=true,       
   linkcolor=blue,          
    citecolor=blue,        
    filecolor=black,      
    urlcolor=blue}

\title[A solution to the YTD Conjecture through Special Fujita Approximations]{A solution to the Yau-Tian-Donaldson Conjecture \\ through Special Fujita Approximations}
\date{} 
\author{Antonio Trusiani  \\
{\tiny Appendix by S\'ebastien Boucksom}}

\address{Università di Roma Tor Vergata\\
Via della Ricerca Scientifica 1, 00133\\
Roma, Italy}
\email{\href{mailto:trusiani@mat.uniroma2.it}{trusiani@mat.uniroma2.it}}
\urladdr{\href{https://sites.google.com/view/antonio-trusiani/home}{https://sites.google.com/view/antonio-trusiani/home}}

 \address{Institut de Math\'ematiques  de Jussieu - Paris Rive Gauche  \\
 Sorbonne Université - Campus Pierre et Marie Curie \\
4 place Jussieu  \\
75252 Paris Cedex 05, France \\}

\email{\href{mailto:sebastien.boucksom@imj-prg.fr }{sebastien.boucksom@imj-prg.fr }}
\urladdr{\href{http://sebastien.boucksom.perso.math.cnrs.fr}{http://sebastien.boucksom.perso.math.cnrs.fr}}

\begin{document}

\maketitle
 
\begin{abstract}
    We show that any big line bundle on a smooth projective variety admits a \emph{special} Fujita approximation: the volume and the first Riemann-Roch coefficient are both approximated by those of ample $\Q$-line bundles on higher models. Exploiting previous works by Boucksom, Jonsson and Li, we solve the Boucksom-Jonsson Regularization Conjecture on the Non-Archimedean entropy functional. As a main consequence, we obtain a solution to the (uniform version of the) Yau-Tian-Donaldson Conjecture: a polarized smooth projective variety $(X,L)$ admits a cscK metric if and only if it is $\Aut^\circ(X,L)$-uniformly $K$-stable. This extends the known Yau-Tian-Donaldson correspondence for smooth Fano varieties.
\end{abstract}

\section{Introduction}

A central theme in Complex Geometry is to relate the existence of special metrics to algebro-geometric notions, and the following solution to the so-called Yau-Tian-Donaldson Conjecture (\cite{Yau87, Tian97, Don02}) represents one of the most classical and studied instances of such phenomena.
\begin{mainthm}[{Solution to the Yau-Tian-Donaldson Conjecture}]\label{thmA}
    Let $L\to X$ be an ample line bundle on a smooth projective variety. 
    Then the following are equivalent:
    \begin{itemize}
        \item[i)] there exists a constant scalar curvature K\"ahler (cscK) metric in $c_1(L)$;
        \item[ii)] $(X,L)$ is $\Aut^\circ(X,L)$-uniformly $K$-stable.
    \end{itemize}
    In particular if $\Aut^\circ(X,L)$ is trivial then the following are equivalent:
    \begin{itemize}
        \item[i)] there exists a unique cscK metric in $c_1(L)$;
        \item[ii)] $(X,L)$ is uniformly $K$-stable. 
    \end{itemize}
\end{mainthm}
Some comments are in order. We first recall that $\Aut^\circ(X,L)$ indicates the identity component of the automorphism group of $(X,L)$. Modulo the fiberwise action, $\Aut(X,L)$ reduces to the \emph{linear automorphism group} $\mathrm{LAut}(X)$~\cite{Fuj78}. In the second part of Theorem~\ref{thmA} the triviality of $\Aut^\circ(X,L)$ is intended as $\mathrm{LAut}(X)=\{\mathrm{Id}\}$. Following~\cite{Tian97, Don02}, $K$-stability then requires that for any (normal) \emph{ample test configuration} $(\cX,\cL)$ the \emph{Donaldson-Futaki invariant} $\DF(\cX,\cL)$ is non-negative and equal to zero exactly when $(\cX,\cL)=(X\times \P^1,p_1^*L+c\cX_0)$ for $c\in\Q$, i.e. when $(\cX,\cL)$ is trivial~\cite{Sto11, Oda15, BHJ17}. Instead, the \emph{uniform} $K$-stability holds if $\DF(\cX,\cL)\geq \delta \jj^\NA(\cX,\cL)$ for a uniform constant $\delta$, where $\jj^\NA(\cX,\cL)$ is a \emph{minimum norm} that measures the triviality of ample test configurations~\cite{Sze15, Der16, BHJ17}. As any non-trivial cocharacter $\C^*\to \Aut(X,L)$ defines a non-trivial product test configuration whose Donaldson–Futaki invariant agrees, up to normalization, with the classical Futaki invariant of the induced \(\mathbb C^*\)-action~\cite{Fut83},
in the general case with automorphisms the notion of $K$-stability and its uniform version must be refined. In this context the $\Aut^\circ(X,L)$-uniform $K$-stability of Theorem~\ref{thmA} represents the weakest uniform notion. It only deals with $\Aut^\circ(X,L)$\emph{-equivariant} ample test configurations and asks that their Donaldson-Futaki invariants dominate a uniform multiple of an equivariant version of the $\jj^\NA$-invariant that considers all the possible twists induced by cocharacters $\C^*\to \Aut(X,L)$~\cite{His16b, His19, Li22, Li22b}. Observe that by definition $\Aut^\circ(X,L)$-uniform $K$-stability requires $\Aut^\circ(X,L)$ to be reductive, which is a well-known necessary condition for the existence of cscK metrics~\cite{Lic57, Mat57}. It is worth stressing that Theorem~\ref{thmA} solves the version of the Yau-Tian-Donaldson Conjecture given in~\cite[Conj.~2.26]{Li22b}, but Theorem~\ref{thmA} can also be stated for a uniform version of $K$\emph{-polystability}, as explained below in Corollary~\ref{corB}. 
\smallskip

Let us now discuss previous results regarding Theorem~\ref{thmA}. The implications $(i)\Rightarrow (ii)$ are known and are derived from the combination of many works, as we are going to briefly describe. Following previous works~\cite{Mab87, Tian00, BB17, BBEGZ19,  Dar15, DarRub17, BDL17, BDL20}, in~\cite{CC21I,CC21II} the existence of cscK metrics has been characterized in terms of the \emph{$\jj$-coercivity} of the \emph{Mabuchi functional} in a certain metric completion $\mathcal{E}^1$ of the space of K\"ahler metrics representing $c_1(L)$. Then the Donaldson-Futaki invariant and the (equivariant) $\jj^\NA$-functional of an (equivariant) ample test configuration $(\cX,\cL)$ can be essentially and respectively obtained as the slope of the (equivariant) Mabuchi functional and of the (equivariant) $\jj$-functional along a \emph{geodesic ray} in $\mathcal{E}^1$~\cite{PS06, PS07, RWN14} associated to $(\cX,\cL)$~\cite{BHJ19}.

Note that the existence of cscK metrics in $c_1(L)$ implies the $K$-polystability of $(X,L)$~\cite{Sto09, Mab08}. However, the converse does not seem to be expected as a consequence of~\cite{ACGTF08, Hat21}. In this perspective, the uniform version of Theorem~\ref{thmA} is suspected to be optimal.

Let us also recall that Theorem~\ref{thmA} is known for smooth \emph{spherical} polarized varieties. In the toric case it follows from combining~\cite{His16b} and~\cite{CC21I, CC21II}. In the general spherical case any \emph{big} test configuration is a Mori dream space (see the appendix of~\cite{Del23}), thus $\Aut^\circ(X,L)$-uniform $K$-stability is equivalent to the analogous notions for \emph{models} and~\cite{Li22b} gives the existence of cscK metrics.

The picture is clearer when $L$ is proportional to the canonical bundle, and Theorem~\ref{thmA} is already known in such cases. Indeed, if $L=K_X$, then there are no non-trivial vector fields, a cscK metric in $c_1(K_X)$ is a K\"ahler-Einstein metric, and it always exists, as proved in~\cite{Yau78, Aub78}. Moreover, $K$-stability holds~\cite{OS15}. In the Fano case $L=-K_X$, cscK metrics in $c_1(X)$ are positively curved K\"ahler-Einstein metrics. In this setting the Yau-Tian-Donaldson Conjecture with respect to $K$-polystability has been established in~\cite{CDS15I, CDS15II, CDS15III, Tian15} (see also~\cite{DS16, Sze16, CSW18, Zhang21}). Furthermore the aforementioned implication from the existence of K\"ahler-Einstein metrics to $K$-polystability has been also extended in~\cite{Berm16} to singular Fano varieties. Then the uniform version for Fano (not necessarily smooth) varieties of Theorem~\ref{thmA} has been shown in~\cite{Li22}, relying on previous works~\cite{BBJ15, LTW17}.

Let us also underline that the $K$-stability notion has turned out to be extremely precious in Algebraic Geometry in the last decades; for instance, it has been tremendously advantageous in the construction of moduli spaces of Fano varieties. This aspect is beyond our purposes, and we invite interested readers to look at~\cite{Xu25} and the references therein.

\subsection*{Special Fujita Approximations}
As explained below, Theorem~\ref{thmA} will be a consequence of the following main result regarding \emph{Fujita Approximations}. 

\begin{mainthm}[{Existence of Special Fujita Approximations}]\label{thmB}
    Let $L\to X$ be a big line bundle on an $n$-dimensional smooth projective variety. Then there are birational morphisms\footnote{The morphisms $p_m$ can be chosen to be given by a sequence of blowups along smooth centers.} $p_m:Y_m\to X$, ample $\Q$-line bundles $L_m$ and effective $\Q$-divisors $E_m$ such that
    \begin{itemize}
        \item[i)] $p_m^*L=L_m+E_m$ for any $m\in\N$;
        \item[ii)] $L_m^n\longrightarrow \Vol(L)$ as $m\to +\infty$;
        \item[iii)] $L_m^{n-1}\cdot K_{Y_m}\longrightarrow \langle L^{n-1}\rangle\cdot K_X$ as $m\to +\infty$. 
    \end{itemize}
    Moreover, if $L$ is $\G$-linearized with respect to a reductive Lie group $\G$ then the morphisms $p_m: Y_m\to X$ can be chosen to be $\G$-equivariant, the ample $\Q$-line bundles $L_m$ to be $\G$-linearized\footnote{We mean that there exists $r\in \N$ such that $rL_m$ is $\G$-linearized as ample line bundle.} and the effective $\Q$-divisors $E_m$ to be $\G$-invariant.
\end{mainthm}
We have denoted by $\langle L^{n-1}\rangle$ the \emph{positive intersection product}~\cite{BDPP13, BFJ09, BouThesis}.
Given a big line bundle on a smooth projective variety $X$, it is always possible to find $p_m:Y_m\to X$, $L_m$ and $E_m$ as in Theorem~\ref{thmB} such that $(i)$ and $(ii)$ hold: this is the statement of classical Fujita Approximations~\cite{Fuj94}, although Fujita originally only required the $\Q$-line bundles $L_m$ to be semiample. The key novelty of Theorem~\ref{thmB} is the convergence of the \emph{first Riemann-Roch coefficients} (by an abuse of nomenclature) appearing in $(iii)$. Any Fujita Approximation satisfying $(iii)$ will be called \emph{Special}, and Theorem~\ref{thmB} shows their existence as conjectured in~\cite[Conj.~4.7]{Li21}. It is worth noting that the Semiample Fujita Approximation given by blowing-up the base ideals of $L\to X$, i.e. the most canonical one, will be proved to be Special within the proof of Theorem~\ref{thmB}.

Observe that since the pullback of big line bundles is still a big line bundle, a version of Theorem~\ref{thmB} for singular projective varieties immediately follows from Hironaka's resolution of singularities.

Fujita Approximations have been extremely useful in tackling key questions in modern Algebraic Geometry. Similarly, the refinement of Theorem~\ref{thmB} might be a strong ally in future works as well.

\subsection*{How Theorem~\ref{thmB} implies Theorem~\ref{thmA} via a solution to the Boucksom-Jonsson Regularization Conjecture}

Theorem~\ref{thmB} solves~\cite[Conj.~4.4,~4.7]{Li21}. In particular, as explained in~\cite[Lem.~4.5]{Li21}, it also provides a solution to the \emph{Boucksom-Jonsson Regularization Conjecture}~\cite[Conj.~2.5]{BJ18} and to the Yau-Tian-Donaldson Conjecture of Theorem~\ref{thmA} (cf.~\cite[Prop.~6.7]{Li22b}). More precisely, following Boucksom-Jonsson's Non-Archimedean approach to $K$-stability~\cite{BJ22, BJ18, BJ25I, BJ23II}, in~\cite{Li22b} Li proved that a strengthened notion of $\Aut^\circ(X,L)$-uniform $K$-stability implies the equivariant coercivity of the Mabuchi functional, and hence the existence of cscK metrics in $c_1(L)$ by~\cite{CC21I, CC21II}. Such a stronger notion is called $\Aut^\circ(X,L)$-uniform $K$-stability \emph{for models} and, roughly speaking, it also considers \emph{big} test configurations (see also~\cite{Li21, DR26}). In particular, a strategy for proving Theorem~\ref{thmA} consists of checking that $\Aut^\circ(X,L)$-uniform $K$-stability implies the analogous stability for models.
As observed in~\cite{Li22b} (see also~\cite[Lem.~4.5]{Li21}), such an implication would hold if one solved the so-called Boucksom-Jonsson Regularization Conjecture. To explain the latter, we recall that the set of ample test configurations $\mathcal{H}^\NA$, interpreted as positive Non-Archimedean metrics, can be endowed with a natural Darvas metric whose metric completion $\mathcal{E}^{1,\NA}$ admits a meaningful description in terms of \emph{maximal geodesic rays}~\cite{DL20, BBJ15, Reb21}.
This mimics the Archimedean picture of the Mabuchi space $\mathcal{H}$ of K\"ahler metrics in $c_1(L)$  (see~\cite{Mab87, Sem92, Don99}) and its metric completion $\mathcal{E}^1$ with respect to the Darvas metric~\cite{Dar15}.
The Boucksom-Jonsson Regularization Conjecture then requires that the \emph{Non-Archimedean entropy} $\hh^\NA(\phi)$ of an element $\phi\in \mathcal{E}^{1,\NA}$ can be continuously approximated by the Non-Archimedean entropies $\hh^\NA(\phi_m)$ of a sequence $\{\phi_m\}_{m\in \N}\subset \mathcal{H}^\NA$ such that $\phi_m\to \phi$.
The Archimedean analogue of this regularizing process has been proved in~\cite[Lem.~3.1]{BDL17}. Equivariant notions of these settings are also known. As a consequence of~\cite[Prop.~6.3]{Li22b}, the desired regularization of the Non-Archimedean entropy can be reduced to the case when $\phi\in\mathcal{E}^{1,\NA}$ is associated to a \emph{model}, i.e. to a big test configuration $(\cX,\cL)$ that is dominating and simple normal crossing\footnote{$(\cX,\cX_0^\mathrm{red})$ simple normal crossing and the morphism $\pi:\cX\to \P^1$ $\C^*$-equivariantly factorizes as $\cX\overset{\pi}{\to} X\times \P^1\overset{p_2}{\to} \P^1$.}. For big models $\hh^\NA(\cX,\cL)=\langle\cL^n \rangle\cdot K_{\cX/X\times \P^1}$, up to base change, i.e. it can be essentially expressed in terms of the first Riemann-Roch coefficient of the big line bundle $\cL\to \cX$~\cite[Prop.~3.2]{Li21}.
Thus Theorem~\ref{thmB} provides the right framework to solve the Boucksom-Jonsson Regularization Conjecture and hence the Yau-Tian-Donaldson Conjecture of Theorem~\ref{thmA}.

\begin{maincor}[{Solution to the Boucksom-Jonsson Regularization Conjecture}]\label{corA}
    Let $\G\subset \Aut^\circ(X,L)$ be a reductive Lie group and let $(\cX,\cL)$ be a $\G$-equivariant model of $(X,L)$ that is dominating and simple normal crossing. Then there exists a sequence of $\G$-equivariant smooth ample test configurations $(\cX_m,\cL_m)$ dominating $(\cX,\cL)$ such that
    $$
    \hh^\NA(\cX_m,\cL_m)\longrightarrow \hh^\NA(\cX,\cL)
    $$
    as $m\to +\infty$.
    In particular the equivariant version of Boucksom-Jonsson Regularization Conjecture~\cite[Conj.~2.5]{BJ18} holds.
\end{maincor}
As mentioned before, in the proof of Theorem~\ref{thmB} we proved that blowing-up base ideals produces Special Semiample Fujita Approximations. This provides a solution to~\cite[Conj.]{BJ23II}, i.e. to a more precise version of the Boucksom-Jonsson Regularization Conjecture.

In the Appendix by S. Boucksom, this entropy approximation conjecture has moreover been established over the field of Laurent series $k\lau{t}$ using the Non-Archimedean formalism and the key Proposition \ref{thm:Discrepancies_Bouck} of this paper.
\smallskip

As explained in~\cite{Li22b}, Corollary~\ref{corA} also implies~\cite[Conj.~1.6]{Li22b}, proving that the slope at infinity of the Archimedean entropy along a maximal geodesic ray coincides with the Non-Archimedean entropy of the associated element in $\mathcal{E}^{1,\NA}$ (this last result has also been recently obtained in~\cite[Thm.~C]{BJ26} and independently in~\cite{DZ26}).

\subsection*{Outline of the proof of Theorem~\ref{thmB}}
By a perturbation argument, we show that it is enough to prove a version of Theorem~\ref{thmB} where the $\Q$-line bundles $L_m$ are only assumed to be \emph{semiample}.

Thus, since the point $(iii)$ holds if and only if $L_m^{n-1}\cdot K_{Y_m/X}\to 0$ (Lemma~\ref{lem:Special_H}, see also~\cite[Lem.~4.10]{Li21}), the problem of \emph{bounding the discrepancies} of the morphisms $p_m$ naturally arises.
Already in~\cite[Lem.~4.10]{Li21} it was observed that if $K_{Y_m/X}\leq C E_m$ for a uniform constant $C>0$, then $L_m^{n-1}\cdot K_{Y_m/X}\to 0$ thanks to the orthogonality estimate of~\cite[Thm.~4.1]{BDPP13}. Note that in~\cite{Li21} a deep study of the geometry in Nakayama's examples of big line bundles that do not admit a birational Zariski decomposition~\cite{Nak04} was performed to get such a bound $K_{Y_m/X}\leq CE_m$.

The key step in our proof is given by Proposition~\ref{thm:Discrepancies_Bouck}, in which we establish the upper bound
\begin{equation}
    \label{eqn:IntroThmDiscre}
    \ord_F K_{Z/X}\leq 2 n\left(\ord_F \mathfrak{a}-\ord_F \mathcal{J}(\mathfrak{a})\right)
\end{equation}
for $F\subset Z$ prime divisor, where $p:Z\to X$ is the normalized blow-up of a coherent ideal sheaf $\mathfrak{a}$. Here $\mathcal{J}(\mathfrak{a})$ denotes the multiplier ideal sheaf associated. Proposition~\ref{thm:Discrepancies_Bouck} is of independent interest and it is obtained by exploiting the properties of the \emph{Jacobian ideals of $\mathfrak{a}$} introduced in~\cite{ELSV04}; indeed, a finer estimate in terms of these ideals can also be deduced. An algebraic proof of Proposition \ref{thm:Discrepancies_Bouck}, actually valid over any algebraically closed field of characteristic $0$, due to S. Boucksom can be found in Proposition \ref{prop:key}.

Considering the base ideals $\mathfrak{a}_m:=\mathfrak{b}\left(|mL|\right)$ and taking common log resolutions of $\mathfrak{a}_m$ and $\mathcal{J}(\mathfrak{a}_m)$,~\eqref{eqn:IntroThmDiscre} leads to
\begin{equation*}
    L_m^{n-1}\cdot K_{Y_m/X}\leq 2 n m L_m^{n-1}\cdot\left(E_m-\tilde{E}_m\right)
\end{equation*}
where $m\tilde{E}_m$ is the effective divisor such that $\mathcal{O}_{Y_m}(-m\tilde{E}_m)=p_m^{-1}\mathcal{J}(\mathfrak{a}_m)\cdot \mathcal{O}_{Y_m}$. Comparing with the previously mentioned sufficient condition described in~\cite{Li21}, we replaced $E_m$ with the difference $E_m-\tilde{E}_m$, where $\tilde{E}_m$ are \emph{closely related} to $E_m$, but we have an extra multiplicative constant $m$.

Next, we exploit classical results on the graded sequence of base ideals and on their associated asymptotic multiplier ideal sheaves to show the existence of a coherent ideal sheaf $\mathfrak{c}$ such that
\begin{equation*}
    \mathcal{J}(\mathfrak{a}_m)\cdot \mathfrak{c} \subset \mathfrak{a}_m
\end{equation*}
for any $m\in\N$. More precisely, reducing to the case when $L=A+E$ for $A$ ample $\Q$-line bundle and $E$ effective $\Q$-divisor such that $\mathbf{B}_+(L)=\mathrm{Supp}(E)$, we show that the ideal sheaf $\mathfrak{c}$ can be taken of the type $\mathcal{O}_X(-m_0E)$ for a constant $m_0\in\N$. Hence
$$
L_m^{n-1}\cdot K_{Y_m/X}\leq 2 n m_0 L_m^{n-1}\cdot p_m^*E,
$$
and the orthogonality property of~\cite[Thm.~4.15,~Thm.~B]{BFJ09},~\cite[Thm.~C]{ELMNP09} yields the desired convergence $\lim_m L_m^{n-1}\cdot K_{Y_m/X}=0$.

\subsection*{The weighted relative setting.}
Recently Lahdili \cite{Lah19} (and independently Inoue~\cite{Ino22}) introduced a useful weighted relative formalism to study several special metrics such as Calabi's extremal metrics and K\"ahler-Ricci solitons; they are all \emph{weighted extremal} metrics. More precisely, for our purposes, let $L\to X$ be an ample line bundle, let $T\subset \Aut^\circ(X,L)$ be a maximal compact torus, let $m_\om:X\to \mathfrak{t}^\vee$ be a moment map for a $T$-invariant K\"ahler form $\om\in c_1(L)$, let $v,w\in C^\infty(\mathfrak{t}^\vee)$ and assume that $v,w$ are positive on the moment polytope $P=m_\om(X)$. We refer to~\cite{Lah19, AJL23, BJT26, PT26, DJL24, DJL25, HL25a} for the precise definitions and for a deep study of this enlarged setting. Here we recall that for fixed \emph{weights} $(v,w)$ there is a translation invariant weighted relative Mabuchi functional $M_{v,w}^{\mathrm{rel}}$ which is the Euler-Lagrange functional of the $(v,w)$-weighted extremal K\"ahler metric. For instance, the unweighted case $(v,w)=(1,1)$ returns to the study of Calabi's extremal K\"ahler metrics, and in particular to the search for cscK metrics when the so-called Futaki invariant vanishes.

In~\cite{AJL23} it has been established that the existence of a $(v,w)$-weighted extremal K\"ahler metric in $c_1(L)$ implies the $T_\C$-coercivity of the weighted relative Mabuchi functional $M_{v,w}^{\mathrm{rel}}$, and furthermore it yields the $T_\C$\emph{-uniform $(v,w)$-weighted relative $K$-stability} of $(X,L)$ (see also~\cite{Ssz11,Lah23, CLS14}). The latter stability notion is essentially analogous to the $T_\C$-uniform $K$-stability, with the Donaldson-Futaki invariant replaced by its $(v,w)$-weighted relative analogue.
In the other direction, recently Han-Liu in~\cite{HL25b} (see also~\cite{Has25}) proved that if $(X,L)$ is $T_\C$-uniformly $(v,w)$-weighted relative $K$-stable \emph{for models}\footnote{As $\T$ is maximal, such notion is equivalent to the \emph{uniform $T$-equivariant relative weighted $\widehat{K}$-polystability} recently considered in~\cite{BJ26}. Indeed, this follows from~\cite[Lem.~8.17]{BJ26} and the fact that the identity component of the \emph{centralizer} of $T\subset \Aut^\circ(X,L)$ is equal to $\T_\C$ when it is assumed to be reductive (as in the definition of uniform $T$-equivariant relative weighted $K$-polystability).} then $M_{v,w}^{\mathrm{rel}}$ is $\T_\C$-coercive. The latter, in turn, implies the existence of a $(v,w)$-weighted extremal K\"ahler metric in $c_1(L)$ when the weight $v$ is log-concave, thanks to~\cite{AJL23, He19, HL25a, DJL24, DJL25}.
\smallskip

We can apply the existence of Special Ample Fujita Approximations from Theorem~\ref{thmB} to this weighted relative setting to show that the $T_\C$-uniform $(v,w)$-weighted relative $K$-stability for models is equivalent to the original one, obtaining the following Yau-Tian-Donaldson correspondence.
\begin{mainthm}\label{thmC}
    Let $L\to X$ be an ample line bundle, let $T\subset \Aut^\circ(X,L)$ be a maximal compact torus, let $\om\in c_1(L)$ be a $T$-invariant K\"ahler form, let $m_\om:X\to \mathfrak{t}^\vee$ and let $v,w\in C^\infty(t^\vee)$ that are positive on the moment polytope $P=m_\om(X)$. Then the following are equivalent:
    \begin{itemize}
        \item[(i)] the $(v,w)$-weighted relative Mabuchi functional is $T_\C$-coercive;
        \item[(ii)] $(X,L)$ is $T_\C$-uniformly $(v,w)$-weighted relative $K$-stable.
    \end{itemize}
    In particular if $v$ is log-concave then the following are equivalent:
    \begin{itemize}
        \item[(i)] there exists a $(v,w)$-weighted extremal K\"ahler metric in $c_1(L)$;
        \item[(ii)] $(X,L)$ is $T_\C$-uniformly $(v,w)$-weighted relative $K$-stable.
    \end{itemize}
\end{mainthm}
Considering the unweighted case $(v,w)=(1,1)$, Theorem~\ref{thmC} in particular provides a Yau-Tian-Donaldson correspondence for Calabi's extremal K\"ahler metrics (see also~\cite{Jub23, JY26}).

\subsection*{On the recent Boucksom-Jonsson and Darvas-Zhang papers \& equivalent stability notions}
In addition to the already cited papers, two interesting articles on solutions to the Yau-Tian-Donaldson Conjecture have recently appeared:~\cite{DZ26, BJ26} (see also~\cite{BJ26b}). 
The current manuscript is independent of them. As mentioned above, they solved~\cite[Conj.~1.6]{Li22b}, establishing a version of Theorem~\ref{thmA} where the uniform $K$-stability notions are replaced by stronger stability notions. Moreover, in both papers, the existence of cscK metrics has been linked to algebro-geometric notions given as variants of the original one.

In~\cite[Thm.~1.1]{DZ26} a Yau-Tian-Donaldson correspondence has been achieved in terms of uniform $K^\beta$-stability for some $\beta>0$, where the latter is a \emph{transcendental quantized} version of uniform $K$-stability, following the analytic picture of the space of K\"ahler metrics. $K^\beta$-stability takes into consideration a twisted version of the Donaldson-Futaki invariant for ample test configurations, which pointwise increasingly approximates the original Donaldson-Futaki invariant. Moreover, they also made progress in the Boucksom-Jonsson Regularization Conjecture of Corollary~\ref{corA}, showing that it was enough to consider a special class of big models (\emph{\textquotedblleft log discrepancy models\textquotedblright}), yielding a version of Theorem~\ref{thmA} with respect to $\Aut^\circ(X,L)$-uniform $K$-stability for log discrepancy models.

In~\cite{BJ26}, following their deeply developed Non-Archimedean formalism, the authors instead showed a Yau-Tian-Donaldson correspondence with respect to (uniform) $\widehat{K}$-polystability. Interestingly, $\widehat{K}$-polystability is proved to be equivalent to its uniform version. $\widehat{K}$-polystability extends the definition of $K$-polystability to all the elements in $\mathcal{E}^{1,\NA}$, not only for those given by ample test configurations. However, although it regards the whole space $\mathcal{E}^{1,\NA}$, $\widehat{K}$-polystability has a concrete algebro-geometric formulation in terms of divisorial valuations, as explained in~\cite[Sec.~8]{BJ26} (cf.~\cite{BJ23II}). Replacing $\mathcal{E}^{1,\NA}$ with $\mathcal{H}^{\NA}$ would give a \emph{uniform $K$-polystability} notion, as the authors have observed. Lastly, in~\cite{BJ26}, a more general study of the weighted relative case has also been performed, connecting the existence of weighted relative extremal metrics to a weighted relative version of the (equivariant) $\widehat{K}$-polystability.
\smallskip

As a consequence of Theorem~\ref{thmB}, we can prove the following result, which in particular unifies these algebro-geometric notions.
\begin{maincor}\label{corB}
    Let $L\to X$ be an ample line bundle on a smooth projective variety. Then the following are equivalent:
    \begin{itemize}
        \item[i)] $(X,L)$ is $\Aut^\circ(X,L)$-uniformly $K$-stable;
        \item[ii)] $(X,L)$ is uniformly $K$-polystable;
        \item[iii)] $(X,L)$ is $\Aut^\circ(X,L)$-uniformly $K$-stable for log discrepancies models;
        \item[iv)] $(X,L)$ is $\Aut^\circ(X,L)$-uniformly $K$-stable for models;
        \item[v)] $(X,L)$ is uniformly $\widehat{K}$-polystable;
        \item[vi)] $(X,L)$ is $\widehat{K}$-polystable.
    \end{itemize}
    In the case $\Aut^\circ(X,L)$ is trivial, we also have that they are all equivalent to
    \begin{itemize}
        \item[vii)] $(X,L)$ is uniformly $K^\beta$-stable for some $\beta>0$.
    \end{itemize}
\end{maincor}
In particular, Theorem~\ref{thmA} can also be stated with respect to uniform $K$-polystability, providing an equivalent version of the Yau-Tian-Donaldson correspondence.

We refer to~\cite{BJ26, DZ26} for the original definitions in $(ii), (iii), (v), (vi)$ and $(vii)$. Combining Theorem~\ref{thmA},~\cite[Thm.~1.1,~Thm. 1.3]{DZ26},~\cite[Thm.~A]{BJ26} and other previous results, the unique implication we prove in Corollary~\ref{corB} is $(ii)\Rightarrow (v)$. As anticipated, this is again a consequence of Theorem~\ref{thmB}.

Observe that the equivalences $(i)\Leftrightarrow (iii) \Leftrightarrow (iv)$ and $(ii)\Leftrightarrow (v)$ are proved algebraically. It is then natural to wonder if there is an algebraic proof connecting them, for instance by checking that $(i)\Leftrightarrow (ii)$. Similarly, an algebraic proof linking $(vi)$ (and also $(vii)$, in the trivial automorphism case) to the others does not seem to be available yet, as already observed in~\cite{BJ26}.

It is also worth underlining that $(i), (ii)$ are \emph{classical} $K$-stability notions since they deal with the Donaldson-Futaki invariants of ample test configurations. In the setting of discrete automorphisms they both coincide with uniform $K$-stability. In such cases, $(vii)$ also considers ample test configurations. However, it deals with a $\beta$-twisted version of Donaldson-Futaki invariants, and to go back to the classical $K$-stability one would need to take the limit as $\beta\to +\infty$.
\smallskip

Note that, as a consequence of Theorem~\ref{thmA} and~\cite[Cor.~1.2]{DR24}, the existence of a unique cscK metric is also detected by \emph{uniform arc $K$-stability} (\cite{DR24,Der26}; see also~\cite{Don12}). Interestingly, this Yau-Tian-Donaldson correspondence has been exploited in~\cite{Der26}, where the author has studied the constant scalar curvature K\"ahler locus for flat families, obtaining new examples of cscK metrics.

\smallskip
Combining Theorem~\ref{thmC} with~\cite[Thm.~B]{BJ26}, one can also deduce an analogous result of Corollary~\ref{corB} for the weighted relative setting.

\subsection*{Other future possible developments} 
In the companion paper~\cite{Tru26} a \emph{transcendental version} of Theorem~\ref{thmB} is investigated, obtaining conditions that ensure the existence of Special Fujita Approximations for general K\"ahler classes on K\"ahler manifolds. As a consequence, some progress on the transcendental Yau-Tian-Donaldson Conjecture is made (following~\cite{DR17, Sjo18, Der18, Sjo20, MP25, MPWN25}).

It is also natural to try to generalize Theorem~\ref{thmA} to singular varieties. Although there are still many challenges, recent progress on singular cscK metrics has been made (see~\cite{PTT23, BJT26, Sze25, PT26}) and on the Yau-Tian-Donaldson Conjecture (see~\cite{HL25b}, following Boucksom-Jonsson's Non-Archimedean formalism). Similarly to the Fano setting, obtaining a singular version of the Yau-Tian-Donaldson correspondence may lead to the construction of moduli spaces and other related topics (see~\cite{FS90, DN25} in this perspective). Note that, as previously explained, a version of Theorem~\ref{thmB} for singular varieties can be easily stated.

\subsection*{Structure of the paper}
Section~\ref{sec:2} is dedicated to introducing all the necessary preliminaries. Although we could have been more concise, we have chosen to recall a few facts regarding $K$-stability as a courtesy to the reader.
Proposition~\ref{thm:Discrepancies_Bouck} is then proved in Section~\ref{sec:3}. As mentioned before, this bound on the discrepancies represents the key tool in proving Theorem~\ref{thmB}.
In Section~\ref{sec:4} Special Fujita Approximations are defined, while Section~\ref{sec:5} is dedicated to the proof of Theorem~\ref{thmB}. In the last Section~\ref{sec:6} we describe in more detail how Theorem~\ref{thmB} implies Theorem~\ref{thmA}, Corollaries~\ref{corA} and~\ref{corB}, and Theorem~\ref{thmC}. Finally, in Appendix~\ref{appendix} by S. Boucksom, an algebraic proof of Proposition \ref{thm:Discrepancies_Bouck} has been provided, also establishing the Regularization Conjecture over the field $k\lau{t}$ through the non-Archimedean formalism.

\subsection*{Disclosure on the use of AI} After the second version of arXiv, an AI-check of the paper found a gap in the proof of Proposition \ref{thm:Discrepancies_Bouck} (but not in the alternative proof included in the Appendix \ref{appendix}). It has also suggested an alternative proof that, if corrected, would have been shorter and would have provided a better multiplicative term. However, the author preferred to correct his original proof, which is reminiscent of the starting idea of the paper. Thus, all the ideas, statements and arguments are the sole responsibility of the author.

\section*{Acknowledgments}
The author is extremely grateful to S. Boucksom, not only for useful mathematical comments but also for suggesting an algebraic alternative proof of Theorem~\ref{thmB} after the first ArXiv version of this manuscript, and for agreeing to add the Appendix~\ref{appendix}.

The author thanks S. Trapani for carefully reading the first drafts of the paper and raising interesting questions that helped improve the article. Finally, the author is also thankful to P. Piccione, M. Jonsson and R. Dervan for mathematical remarks, and to S. Jubert, Y. Chen, F. Fagioli, G. M. Lido, R. Vacca, G. Codogni, F. Bernasconi, L. Franzoi for fruitful discussions during the preparation of the manuscript.

\section{Preliminaries}\label{sec:2}
Let $X$ be a smooth projective variety of (complex) dimension $n$. We recall that a line bundle $L\to X$ is said to be \emph{big} if
$$
\Vol(L):=\limsup_{k\to +\infty}\frac{\dim \mathrm{H}^0(X,kL)}{k^{n}/n!}>0.
$$

\subsection{Positive intersection product}\label{ssec:Pos_Int_Prod}
Let $L\to X$ be a big line bundle. For any $r=1,\dots,n$, the \emph{positive intersection product} $\langle L^r\rangle$ has been introduced analytically in~\cite{BDPP13} (following the PhD thesis~\cite{BouThesis}) and has been algebraically described in~\cite{BFJ09}. For the purposes of the paper, we recall that
$$
\langle L^{n-1}\rangle\cdot D:=\langle c_1(L)^{n-1} \rangle\cdot D=\sup_{\mu^*L =P+E} P^{\, n-1}\cdot \mu^*D
$$
for any $D$ ample divisor, where the supremum is over all birational morphisms $\mu:Y\to X$ from $Y$ smooth projective varieties and all the decompositions $\mu^*L=P+E$ for $P$ nef $\Q$-line bundles and $E$ effective $\Q$-divisors. The definition extends linearly. Note that Fujita's Approximation Theorem~\cite{Fuj94} can be stated as $\langle L^n\rangle=\Vol(L)$.

\subsection{Augmented base locus and Orthogonality}\label{ssec:Non-Kahler}
Let $L\to X$ be a big line bundle. Following~\cite{Nak00, Nak03}, the \emph{augmented base locus of $L$} has been introduced in~\cite[Def.~1.2]{ELMNP06} as
$$
\mathbf{B}_+(L):=\bigcap_{A}\mathbf{B}(L-A)
$$
where the intersection is over all the ample $\Q$-line bundles and where $\mathbf{B}(\cdot)$ is the stable base locus (see~\cite[Def.~2.1.20]{Laz04}). For a big and nef line bundle $L$ the augmented base locus coincides with the \emph{null-locus}, namely
\begin{equation*}
    \mathbf{B}_+(L)=\bigcup_{L^{\dim V} \cdot V=0}V
\end{equation*}
where the union is over all analytic irreducible subsets $V\subset X$ (see~\cite[Thm.~0.3]{Nak00}).

The following orthogonality result will be essential in the proof of Theorem~\ref{thmB}.
\begin{thm}[{\cite[Sec.~4]{BDPP13},~\cite[Thm.~4.15, Thm.~B]{BFJ09},~\cite[Thm.~C]{ELMNP09}}]\label{thm:Orthogonality}
    Let $L\to X$ be a big line bundle on a smooth projective variety $X$. Then
    \begin{equation*}
        \Vol(L)=\langle L^{n-1}\rangle \cdot L.
    \end{equation*}
    In particular a prime divisor $D$ is contained in the augmented base locus $\mathbf{B}_+(L)$ if and only if
    \begin{equation*}
        \langle L^{n-1}\rangle \cdot D=0.
    \end{equation*}
\end{thm}
We refer to~\cite{BDPP13, BFJ09} to see how Theorem~\ref{thm:Orthogonality} is connected to the duality of the psef and the movable cone, to the differentiability of the volume function on the big algebraic cone, and to Demailly's holomorphic Morse inequalities.

\subsection{Graded system of ideals and multiplier ideal sheaves}\label{ssec:Ideals}
We recall that a family of ideal sheaves $\mathfrak{a}_{\bullet}=\{\mathfrak{a}_m\}_{m\in\Z_{\geq 0}}$ is said to be a \emph{graded system of ideals} if $\mathfrak{a}_0=\mathcal{O}_X$ and $\mathfrak{a}_m \cdot \mathfrak{a}_l\subset \mathfrak{a}_{m+l}$ for any $m,l\in\Z_{\geq 0}$ (\cite[Def.~2.4.14]{Laz04}). A typical example is given by the base ideals $\mathfrak{a}_m:=\mathfrak{b}\left(\left| mL\right|\right)$ of a line bundle $L\to X$ on a projective variety~\cite[Def.~1.1.8]{Laz04}.

Given a coherent ideal sheaf $\mathfrak{a}\subset \mathcal{O}_X$ and a coefficient $c\in \Q_{>0}$, the \emph{multiplier ideal sheaf} $\mathcal{J}(\mathfrak{a}^c)$ is defined as
$$
\mathcal{J}(\mathfrak{a}^c)=p_*\mathcal{O}_Y\left(K_{Y/X}-[cE]\right)
$$
where $p:Y\to X$ is any fixed log resolution of the ideal $\mathfrak{a}$ and where $E$ is the effective divisor such that $p^{-1}\mathfrak{a}\cdot\mathcal{O}_Y=\mathcal{O}_Y(-E)$. Given a graded system of ideals $\mathfrak{a}_{\bullet}$ and an exponent $c\in\Q_{>0}$, we denote by $\mathcal{J}\left(\mathfrak{a}_{\bullet}^c\right)$ the unique maximal element of $\left\{\mathcal{J}\big(\mathfrak{a}_m^{c/m}\big)\right\}_{m\geq 1}$. When $\mathfrak{a}_{\bullet}$ is the graded system associated to the base ideals $\mathfrak{a}_m:=\mathfrak{b}\left(\left|mL\right|\right)$ of a line bundle $L\to X$, one obtains that $\mathcal{J}\left(\mathfrak{a}_{\bullet}^c\right)=\mathcal{J}\big(c\cdot \lVert L\rVert\big)$ is the asymptotic multiplier ideal associated to the complete linear series of $L$ (see~\cite[\S 11.1.A]{Laz04II})

We refer to~\cite[Part~Three]{Laz04II} and the references therein for many properties of these sheaves.

\subsection{CscK metrics}\label{ssec:CscK}
A \emph{constant scalar curvature K\"ahler metric} (cscK metric) in a cohomology class $\alpha$ is a K\"ahler metric associated to a K\"ahler form $\om\in \alpha$ such that
$$
    S(\om)\equiv \overline{S}
$$
where $S(\om)=\tr_{\om}\Ric(\om)=n\frac{\Ric(\om)\wedge\om^{n-1}}{\om^n}$, while $\overline{S}=\frac{\int_X S(\om)\om^{n}}{\int_X \om^{n}}=n\frac{c_1(X)\cdot\alpha^{n-1}}{\alpha^n}$.
The typical examples of cscK metrics are the so-called \emph{K\"ahler-Einstein metrics} in which the K\"ahler forms are proportional to their Ricci forms.
In the case $\alpha=c_1(L)$ for an ample line bundle, $(X,L)$ is said to admit a cscK metric if there is a cscK metric with the associated form in $c_1(L)$.
    
   \subsection{(Uniform) $K$-stability}\label{sec:K-stability}
    
    Let $L\to X$ be an ample line bundle on a normal projective variety. We briefly recall here the definition of (uniform) $K$-stability. We refer to~\cite{BHJ17} and references therein.

    A \emph{test configuration} $\cX$ for $X$ is a normal\footnote{It is enough to consider normal test configurations as a consequence of~\cite[Prop.~3.15]{BHJ17}.} projective variety equipped with a $\C^*$-action and with a $\C^*$-equivariant surjective morphism $\pi: \cX\to \P^1$ such that $\pi^{-1} \left(\P^1\setminus \{0\}\right) \simeq X\times\left( \P^1\setminus \{0\}\right)$, where $\P^1$ is endowed with the standard $\C^*$-action. A \emph{test configuration} $(\cX,\mathcal{L})$ for $(X,L)$ then consists of a test configuration $\cX$ for $X$ and of a $\C^*$-linearized $\Q$-line bundle $\mathcal{L}$ lifting the $\C^*$-action on $\cX$ such that $\mathcal{L}_{|\pi^{-1}(\P^1\setminus \{0\})}\simeq p_1^*L_{|X\times (\P^1\setminus\{0\})}$ through the isomorphism $\pi^{-1}(\P^1\setminus\{0\})\simeq X\times (\P^1\setminus\{0\})$, where $p_1$ is the projection on the first factor.

    Two test configurations $(\cX,\cL), (\cX',\cL')$ are said to be \emph{equivalent} if there exists a third test configuration $(\cX'',\cL'')$ and $\C^*$-equivariant morphisms $\pi: \cX''\to \cX,\pi':\cX''\to \cX'$ such that $\cL''=\pi^*\cL= \pi'^*\cL'$. Any test configuration $(\cX,\cL)$ admits an equivalent test configuration $(\cX',\cL')$ that is \emph{dominating}, i.e. such that the map $\pi':\cX'\to \P^1$ admits a $\C^*$-equivariant factorization through a morphism $\cX'\to X\times \P^1$. By Hironaka, such dominating test configuration $(\cX',\cL')$ can also be chosen to be \emph{smooth} (i.e. $\cX'$ smooth) when $X$ is smooth.
    A test configuration $(\cX,\cL)$ is said to be \emph{trivial} if $\cX=X\times \P^1$ with trivial $\C^*$-action. In such a case, we necessarily have $\cL=p_1^* L+c\cX_0$ for $c\in\Q$ where $\cX_0:=\pi^{-1}\{0\}$ with obvious notation~\cite[Lem.~2.10]{BHJ17}.

    A test configuration $(\cX,\cL)$ is \emph{ample} (resp. \emph{semiample}) if $\cL$ is ample (resp. semiample).
    For any semiample test configuration $(\cX,\cL)$ the \emph{Donaldson-Futaki invariant} $\DF(\cX,\cL)$ is naturally attached to the induced action on the central fiber $(\cX_0,\cL_0)$. We refer to~\cite[Sec.~3.3]{BHJ17} for the original definition of $\DF(\cX,\cL)$; here we recall that by~\cite[Prop.~3.12]{BHJ17} it can be intersection-theoretically described as
    \begin{equation}
        \label{eqn:DF}
        \DF(\cX,\cL)=\overline{S}\,\frac{\cL^{n+1}}{n+1}+\cL^n\cdot K_{\cX/\P^1}
    \end{equation}
    where $\bar{S}:=-nK_X\cdot L^{n-1}/L^n $ as in subsection~\ref{ssec:CscK}.
    
    To characterize when a test configuration is trivial, a \emph{minimum norm} has been introduced independently in~\cite{Der16, BHJ17}. More precisely, letting $(\cX,\cL)$ be an ample test configuration, they defined a quantity $\jj^\NA(\cX,\cL)$ which is non-negative and vanishes if and only if $(\cX,\cL)$ is a trivial test configuration~\cite[Thm.~1.3]{Der16},~\cite[Cor.~B]{BHJ17}. Passing to an equivalent semiample test configuration $(\cX',\cL')$ that is dominating through $\mu_\cX:\cX\to X\times \P^1$, we get the formula
    \begin{equation*}
        \jj^{\NA}(\cX,\cL)=\jj^\NA(\cX',\cL')=\mu_\cX^*p_1^*L^n\cdot \cL'-\frac{\cL'^{n+1}}{n+1}.
    \end{equation*}
    \begin{defi}
        A normal polarized variety $(X,L)$ is said to be
        \begin{itemize}
            \item[i)] \emph{$K$-stable} if $\DF(\mathcal{X},\mathcal{L})\geq 0$ for any ample test configuration $(\cX,\cL)$ with equality if and only if $(\cX,\cL)$ is trivial;
            \item[ii)] \emph{uniformly $K$-stable} if there exists $\sigma>0$ such that
            $$
            \DF(\cX,\cL)\geq \sigma \jj^\NA(\cX,\cL)
            $$
            for any ample test configuration $(\cX,\cL)$.
        \end{itemize}
    \end{defi}
    Uniform $K$-stability implies $K$-stability as $\jj^\NA$ vanishes exactly at trivial ample test configurations, as previously mentioned.
    
    \subsection{$K$-polystability and $\G$-uniform $K$-stability}\label{sec:G-uniform}
    Let $\Aut^\circ(X,L)$ be the identity component of the automorphisms of $X$ that lift to the ample line bundle $L$. Any $\C^*$-action on $X$ naturally induces a diagonal $\C^*$-action on $X\times \C$, producing a test configuration $\cX$ for $X$. If such $\C^*$-action gives a $\C^*$-linearization of $L$, i.e. it is a cocharacter $\C^*\to \Aut^\circ(X,L)$, then it induces a test configuration $(\cX,\cL)$ that is $\C^*$-equivariantly isomorphic to $(X,L)\times \C$ over $\C\simeq \P^1\setminus\{+\infty\}$, i.e. $(\cX,\cL)$ is a \emph{product test configuration}. Following the pioneering works~\cite{Don02,Tian97}, in this case the Donaldson-Futaki invariant essentially coincides with the usual Futaki invariant~\cite{Fut83} of the cocharacter (cf.~\cite[Prop.~4.12]{Sjo20}). In particular, (uniform) $K$-stability cannot detect the existence of cscK metrics when $\Aut^\circ(X,L)$ is not trivial, justifying the notions of \emph{$K$-polystability} and \emph{$\G$-uniform $K$-stability}, which we now recall.

    Following~\cite{His16b, His19, Li22}, we fix $\G\subset \Aut^\circ(X,L)$ as a reductive complex Lie group, and let $\T$ be the algebraic torus given as the identity component of the center\footnote{In~\cite{BJ26}, it has been shown how considering the identity component of the \emph{centralizer} of $\G$ yields a finer notion that is more suitable for a Yau-Tian-Donaldson correspondence (see subsection $8.4$ in \emph{loc. cit.}). The two notions clearly coincide when $\G=\Aut^\circ(X,L)$, as in Theorem~\ref{thmA}. Thus, we will not recall the general definition of uniform equivariant $K$-polystability given in~\cite{BJ26}, but we invite the interested reader to look at~\cite{BJ26} for the complete picture.} of $\G$ and denote by $N_\Z(\T):=\Hom(\C^*,\T)$ the cocharacter lattice of $\T$. Set also $N_\Q(\T):=N_\Z(\T)\otimes_\Z\Q$. We then say that a test configuration $(\cX,\cL)$ is $\G$\emph{-equivariant} if $(\cX,\cL)$ is also endowed with a $\G$-action that commutes with the $\C^*$-action and coincides with the action of the first factor on $(X,L)\times \P^1\setminus\{0\}\simeq (\cX,\cL)_{\pi^{-1}(\P^1\setminus\{0\})}$.
    To any $\G$-equivariant test configuration $(\cX,\cL)$ and to any element $\mu\in N_\Q(\T)$, it is naturally associated a $\G$-equivariant test configuration $(\cX_\mu,\cL_\mu)$ where the action is \emph{twisted by $\mu$} (see~\cite{His16b},~\cite[Def.~3.2]{Li22}).
    \begin{defi}
        Let $(X,L)$ be a normal polarized variety and let $\G\subset \Aut^\circ(X,L)$ be a reductive complex Lie group. Then $(X,L)$ is said to be
        \begin{itemize}
            \item[i)] \emph{$K$-polystable} if $\DF(\cX,\cL)\geq 0$ for any ample test configuration $(\cX,\cL)$ with equality if and only if $(\cX,\cL)$ is a product;
            \item[ii)] \emph{$\G$-uniformly $K$-stable} if there exists $\sigma>0$ such that
            $$
            \DF(\cX,\cL)\geq \sigma \inf_{\mu\in N_\Q(\T)}\jj^\NA(\cX_\mu,\cL_\mu)
            $$ 
            for any $\G$-equivariant ample test configuration $(\cX,\cL)$.
        \end{itemize}
    \end{defi}
    The quantity $\jj^\NA(\cX_\mu,\cL_\mu)$ can also be explicitly described (see~\cite[eq.~(2.3)]{His16b}). Note that uniform $K$-stability implies $\G$-uniformly $K$-stability, and that $K$-polystability coincides with $K$-stability when $\Aut^\circ(X,L)$ is trivial. 

    \subsection{$K$-stability for models}\label{sec:K-stab_models}
    There have been several works~\cite{Mab16, Don12, DR24} on strengthening the $K$-stability notion to obtain the existence of cscK metrics. Among them, following~\cite{Sze15}, in~\cite{Li22b, Li21} the $K$\emph{-stability for models} (also called \emph{for model filtrations}) has been introduced and considered for smooth projective varieties. The main idea is to consider (big) test configurations instead of only ample test configurations.
    
    More precisely, since any test configuration $(\cX,\cL)$ is ample on the generic fiber of $\pi:\cX\to \P^1$, there exists a constant $c\gg 1$ such that $(\cX,\cL_c)$ is big where $\cL_c:=\cL+c\cX_0$. Moreover, as a consequence of~\cite[Lem.~A.6]{BFJ15} (see also the proof of~\cite[Thm.~A.4]{BFJ15}) if $c\gg 1$ is large enough then $(\cX,\cL_c)$ is a \emph{big model} in the sense of~\cite[Def.~2.1]{Li21}, i.e. $\mathbf{B}(\cL)=\mathbf{B}(\cL/\P^1)$ where $\mathbf{B}(\cL/\P^1)$ indicates the relative stable base locus. Approximating big models by semiample test configurations and exploiting the Non-Archimedean formalism developed in~\cite{BJ18, BJ25I, BJ23II}, in~\cite{Li21} the formula~\eqref{eqn:DF} has been extended as
    \begin{equation}\label{eqn:DF2}
        \DF(\cX,\cL)=\overline{S}\,\frac{\langle\cL_c^{n+1}\rangle}{n+1}+\langle\cL_c^n\rangle\cdot K_{\cX/\P^1}
    \end{equation}
    for any $(\cX,\cL)$ test configuration where $c\gg 1$ is big enough so that $(\cX,\cL_c)$ is a big model. It can be proven that~\eqref{eqn:DF2} does not depend on the choice of $c\gg 1$. 
    
    Similarly, approximating big models using semiample test configurations shows how the functional $\jj^{\NA}$ naturally extends to big models. Moreover, as in subsection~\ref{sec:G-uniform}, $\G$-equivariant big models $(\cX,\cL)$ can be twisted into big models $(\cX_\mu,\cL_\mu)$.
    We can now recall the extended definitions of stability for models given in~\cite[Def.~2.25]{Li22b}.
    \begin{defi}
        Let $(X,L)$ be a smooth polarized variety and let $\G\subset \mathrm{Aut}^\circ(X,L)$ be a reductive Lie group. Then $(X,L)$ is said to be \emph{$\G$-uniformly $K$-stable for models} if there exists $\sigma>0$ such that
        $$
            \DF(\cX,\cL)\geq \sigma \inf_{\mu\in N_\Q(\T)}\jj^\NA(\cX_\mu,\cL_\mu)
        $$
        for any $\G$-equivariant big model $(\cX,\cL)$. If $\G=\{e\}$ then $\G$-uniform $K$-stability for models is simply said \emph{uniform $K$-stability for models}.
    \end{defi}
    Note that such definitions are equivalent to those given in~\cite{Li22b} as a consequence of~\cite[Prop.~3.4]{Li21}, and that clearly $\G$-uniform $K$-stability for models implies $\G$-uniform $K$-stability.

    \subsection{Boucksom–Jonsson Regularization Conjecture}\label{ssec:RC}
    In~\cite{BJ18,BJ25I, BJ23II} Boucksom and Jonsson deeply investigated a Non-Archimedean approach to $K$-stability, stating their so-called \emph{Regularization Conjecture} (see~\cite[Conj.~2.5]{BJ18}).
    As a consequence of~\cite[Prop.~6.3]{Li22b}, the proof of which is based on~\cite{BFJ15, BFJ16}, this key conjecture can be stated in the following way (see also~\cite[Conj.~4.4]{Li21},~\cite[Conj.~2.12]{Li22c}, which treated the case without automorphisms).
    \begin{conj}\label{conj:RC}
        Let $(X,L)$ be a smooth polarized variety and let $\G\subset \Aut^\circ(X,L)$ be a reductive complex Lie group. For any dominating and simple normal crossing $\G$-equivariant big model $(\cX,\cL)$ there exists a sequence of smooth $\G$-equivariant ample test configurations $\{(\cX_m,\cL_m)\}_{m\in\N}$ dominating $(\cX,\cL)$ such that
        $$
        \lim_{m\to+\infty} \cL_m^{n+1}=\Vol_{\cX}(\cL), \quad \quad \lim_{m\to +\infty}  \cL_m^n\cdot K_{\cX_m}=\langle \cL^n\rangle\cdot K_{\cX}.
        $$
    \end{conj}
    As recalled in the Introduction, solving Conjecture~\ref{conj:RC} would imply that $\G$-uniform $K$-stability is equivalent to $\G$-uniform $K$-stability for models, and it would conclude the proof of the Yau-Tian-Donaldson Conjecture as stated in Theorem~\ref{thmA}. 
    \begin{thm}[{\cite[Thm.~1.10]{Li22b}}]\label{thm:YTD_F}
        Let $(X,L)$ be a smooth projective variety that is $\Aut^\circ(X,L)$-uniformly $K$-stable. If Conjecture~\ref{conj:RC} holds then $(X,L)$ admits a cscK metric.
    \end{thm}

    Note that being $\Aut^\circ(X,L)$-uniformly $K$-stable includes the fact that $\Aut^\circ(X,L)$ is reductive.
    
\section{Upper bounds on discrepancies}\label{sec:3}
In this section $X$ is a smooth projective variety of dimension $n$. The goal of this section is to prove the following result, which will play a key role in proving Theorem~\ref{thmB}.

\begin{prop}\label{thm:Discrepancies_Bouck}
    Let $\mathfrak{a}\subset \mathcal{O}_X$ be a coherent ideal sheaf, denote by $\mathcal{J}(\mathfrak{a})\subset \mathcal{O}_X$ its multiplier ideal sheaf and let $p:Y\to X$ be the normalized blow-up along $\mathfrak{a}$. Then
    \begin{equation}
        \label{eqn:Bouck}
        \ord_E K_{Y/X}\leq 2 n\left(\ord_E \mathfrak{a} - \ord_E \mathcal{J}(\mathfrak{a})\right)
    \end{equation}
    for any prime divisor $E\subset Y$.
\end{prop}
An algebraic proof of this result, due to S. Boucksom, is given in Appendix~\ref{appendix}, where the multiplicative constant is reduced to be $n$. For the more analytic proof given below, we recall that a function $u:X\to \R$ is said to be \emph{quasi-plurisubharmonic} (q-psh) if locally $u\overset{loc}{=}g+\varphi$ for $g\in C^{\infty}$ and $\varphi$ psh. 
Letting $\theta$ be a smooth closed $(1,1)$-form, the set $\PSH(X,\theta)$ denotes all the q-psh functions $u$ such that $\theta_u:=\theta+\ddc u\geq 0$ as $(1,1)$-currents. Here $\dc:= \frac{i}{4\pi}(\bar{\partial}-\partial)$ so that $\ddc=\frac{i}{2\pi}\partial \bar{\partial}$.
\begin{proof}
    The following proof is strongly inspired by the proof of~\cite[Thm.~4.2]{ELSV04}.
    
    Let $\mathfrak{\tilde{a}}\subset \mathcal{O}_X$ be the unique coherent ideal sheaf such that $\codim V(\mathfrak{\tilde{a}})\geq 2$ and such that $\mathfrak{a}=\mathfrak{\tilde{a}}\cdot \mathcal{O}_X(-G)$ for an effective divisor $G$. Then $p:Y\to X$ is the normalized blow-up along $\mathfrak{\tilde{a}}$ and $\mathcal{J}(\mathfrak{a})=\mathcal{J}(\mathfrak{\tilde{a}})\cdot \mathcal{O}_X(-G)$ since $G$ is integral. Thus we can and will suppose that $\codim V(\mathfrak{a})\geq 2$ without loss of generality. Moreover as a consequence of~\cite[Lem.~B.4]{BBJ15} it is enough to prove~\eqref{eqn:Bouck} for $p$-exceptional prime divisors.

    Let $\{U_k\}_k$ be a locally finite open cover of $X$ and let $\{\chi_k\}_k$ be a partition of unity subordinate to the open cover $\{U_k\}_k$. After possibly refining the cover, we may assume that $\mathfrak{a}$ admits local generators $\{f_{k,j}\}_j$ over $U_k$ for any $k$. We then set
    $$
    \varphi:=\log \left(\sum_k \chi_k \sum_j \lvert f_{k,j} \rvert^2 \right),
    $$
    which is a q-psh function. 
    Letting $\om\in c_1(A)$ K\"ahler form for $A$ ample, we can suppose that $\varphi\in \PSH(X,\om)$ unless taking a large multiple of $A$. By construction
    \begin{equation}
        \label{eqn:Above}
            p^*\om_\varphi = \eta + [D]
    \end{equation}
    where $D$ is an effective divisor such that $\mathcal{O}_Y\cdot p^{-1}\mathfrak{a}=\mathcal{O}_Y(-D)$ while $\eta$ is a semipositive closed form representing the divisor $p^*A-D$. Replacing $A$ by a large multiple, we can assume that $\eta\geq ap^*\omega$ for $a>0.$ Similarly, as $-D$ is $p$-ample (see, for instance, \cite[Lem 1.8]{BHJ17}), we can also suppose $\{\eta\}^{n-1}\cdot E>0$ for a fixed $p$-exceptional prime divisor $E$, unless, again, replacing $A$ by a large multiple. Thus, we can find $y\in E\cap Y^{\mathrm{reg}}$ not belonging to any other $p$-exceptional divisor such that $E\overset{loc}{=}\{w_1=0\}$ locally around $y$ for holomorphic coordinates $(w_1,\dots,w_n)$ and such that $\eta_{|E}>0$ locally around $y$. Moreover, supposing without loss of generality that $x=p(y)$ is a smooth point of the subvariety $Z=p(E)_{\redu}$ of codimension $c\geq 2$, we claim that
    \begin{equation}\label{eqn:To_Prove}
        \eta^n\geq C_0 \lvert w_1 \rvert^{\frac{2}{c}(\ord_E K_{Y/X}+1-c})\, d\lambda
    \end{equation}
    locally around $y$ for a constant $C_0>0$, where $d\lambda$ is the Lebesgue measure with respect to the coordinates $(w_1,\dots,w_n)$. To prove \eqref{eqn:To_Prove} we start fixing suitable smooth functions $a_j$ such that $v:=\partial_1+\sum_{j=2}^na_j\partial_j$ is $\eta$-orthogonal to $\partial_j$ for any $j=2,\dots,n$ on a neighbourhood of $y$. Here we clearly set $\partial_j:=\frac{\partial}{\partial w_j}$.
    Thus, we obtain $\eta^n\geq B_1 \eta(v,\bar{v})\, d\lambda$ for a uniform constant $B_1>0$ on a small neighbourhood of $y$, where we used that $\eta_{|E}>0$ and that $\eta$ is smooth. Since $\eta\geq a\,p^*\om$ for $a>0$, we deduce that
    $$
    \eta^n\geq B_1\delta \, p^* \ddc \lvert f \rvert^2(v,\bar{v})\, d\lambda=B_2 \lvert v(f\circ p) \rvert^2\,d\lambda
    $$
    for any fixed holomorphic function $f$ defined on a neighbourhood of $x=p(y)$, where $\delta>0$ is a small constant such that $\delta\ddc \lvert f\rvert^2\leq a\omega$ and where $B_2>0$ is a uniform constant. Possibly changing the point $y$, we can then choose $f\in \mathcal{I}_Z$ such that $f\circ p= w_1^b u$ for $b:=\ord_E \mathcal{I}_Z\geq 1$ and $u$ holomorphic function never vanishing on a neighbourhood of $y$. By an easy computation we then obtain $v(f\circ p)=w_1^{b-1}(bu+w_1u')$ for a smooth function $u'$, from which it follows that $\lvert v(f\circ p) \rvert^2\geq B_3\lvert w_1\rvert^{2(b-1)}$ on a neighbourhood of $y$. 
    To conclude the proof of \eqref{eqn:To_Prove} it then remains to relate $b$ with $\ord_E K_{Y/X}$. But as $x$ is a smooth point in $Z$, by \cite[1.4.(7)]{DK99} we know that $(\sum_j \lvert g_j \rvert^2)^{-(c-\varepsilon)}$ is locally integrable around $x$ for any $\varepsilon>0$, where $(g_j)_j$ are local generators of $\mathcal{I}_Z$. Lifting this integrability on $Y$, it follows that $\lvert w_1\rvert^{2\left(\ord_E(K_{Y/X})-(c-\varepsilon)b\right)}$ is locally integrable around $y$, i.e. $b< (1+\ord_E K_{Y/X})/(c-\varepsilon)$ for any $\varepsilon>0$. Hence \eqref{eqn:To_Prove} follows from $b\leq (1+\ord_E K_{Y/X})/c$ and the inequality $\eta^n\geq B_2B_3 \lvert w_1\rvert^{2(b-1)}$.  

    Next, we fix holomorphic coordinates $(z_1,\dots,z_n)$ around $x=p(y)$, denoting by $\omega_{\mathrm{std}}$ the standard K\"ahler form in such coordinates. We claim that there exists a constant $C_1>0$ such that
    \begin{equation}
        \label{eqn:Claim}
        \omega_\varphi\leq C_1 F^{-1} H \omega_{\mathrm{std}}
    \end{equation}
    over $V\setminus (V\cap V(\mathfrak{a}))$ where $V$ is small open set containing $x=p(y)$ and where
    $$
    F:=\sum_j \lvert f_j \rvert^2, \quad H:=\sum_j \lvert f_j \rvert^2+ \sum_j \sum_{s=1}^n \left\lvert \frac{\partial f_j}{\partial z_s} \right\rvert^2
    $$ for $\{f_j\}_j $ local generators of $\mathfrak{a}$. We first show how the claim~\eqref{eqn:Claim} would conclude the proof.

    Since $\eta=p^*\omega_\varphi$ as $(1,1)$-forms over $U\setminus (U\cap E)$ for a small open set $U$ containing $y$, combining~\eqref{eqn:Above} with~\eqref{eqn:Claim} and~\eqref{eqn:To_Prove} we get
    \begin{equation}
        \label{eqn:Bound}
        C_0\lvert w_1 \rvert^{\frac{2}{c}(\ord_E K_{Y/X}+1-c)}\,d\lambda\leq \eta^n \leq C_1^n (F\circ p)^{-n} (H\circ p)^n p^*\omega_{\mathrm{std}}^n\leq C_2 (F\circ p)^{-n}(H\circ p)^n \lvert  w_1\rvert^{2\ord_E K_{Y/X}}d\lambda
    \end{equation}
    over $U\setminus (U\cap E)$, where we also used the equality $p^*\om^n_{\mathrm{std}}= e^g\lvert w_1\rvert^{2\ord_E K_{Y/X}}\, d\lambda$ for a smooth function $g$ locally around $y$. 
    From~\eqref{eqn:Bound} we deduce that
    $$
    n\ord_E (H\circ p) + \frac{c-1}{c}\ord_E K_{Y/X}-n \ord_E \mathfrak{a}+\frac{c-1}{c}\leq 0
    $$
    as by construction $F\circ p= \lvert w_1\rvert^{2\ord_E \mathfrak{a}}e^h$ for a smooth function $h$. Next, we observe that $\ord_E (H\circ p)=\ord_E \mathrm{Jac}(\mathfrak{a})$ where $\mathrm{Jac}(\mathfrak{a})\subset \mathcal{O}_X$ is the \emph{Jacobian ideal} of $\mathfrak{a}$ (cf.~\cite[Def.~4.1]{ELSV04}). Hence
    \begin{equation}
        \label{eqn:Stronger}
        \ord_E K_{Y/X}\leq \frac{c}{c-1} n\left(\ord_E \mathfrak{a}-\ord_E \mathrm{Jac}(\mathfrak{a})\right)-1\leq 2n \left(\ord_E \mathfrak{a}-\ord_E \mathrm{Jac}(\mathfrak{a})\right),
    \end{equation}
    and the desired inequality~\eqref{eqn:Bouck} follows from the inclusion $\mathrm{Jac}(\mathfrak{a})\subset \mathcal{J}(\mathfrak{a})$ proved in~\cite[Thm.~4.2]{ELSV04}.
    \smallskip
    
    It remains to prove the claim~\eqref{eqn:Claim}. As
    $
    \ddc \log f = f^{-1}\ddc f-f^{-2}\de f\wedge \dc f\leq f^{-1}\ddc f,
    $
    it follows that
    $$
    \ddc \varphi\leq  e^{-\varphi}\sum_k \ddc\left(\chi_k F_k\right)=e^{-\varphi}\sum_k\left(\underbrace{F_k\ddc \chi_k}_{\mathrm{(I)_k}}+ \underbrace{\de F_k \wedge \dc \chi_k+ \de \chi_k\wedge \dc F_k}_{\mathrm{(II)_k}}+ \underbrace{\chi_k \ddc F_k}_{\mathrm{(III)_k}} \right)
    $$
    where $F_k:=\sum_j\lvert f_{j,k} \rvert^2$. Locally around $x$ we can assume $k=1,\dots,M$. Moreover, setting $H_k=\sum_j\lvert f_{j,k} \rvert^2+ \sum_j\sum_s \lvert\frac{\partial f_{k,j}}{\partial z_s}\rvert^2$, the quantities $F/F_k$ and $H/H_k$ are bounded around $x$ for any such $k=1,\dots,M $. Indeed passing from $F$ to $F_k$ (respectively from $H$ to $H_k$) corresponds to a change of local generators of the ideal $\mathfrak{a}$ (resp. of the Jacobian ideal $\mathrm{Jac}(\mathfrak{a})$). Thus to prove~\eqref{eqn:Claim} we separately check that all three terms $\mathrm{(I)}_k$, $\mathrm{(II)}_k$, $\mathrm{(III)}_k$ can be dominated from above by $H_k\omega_{\mathrm{std}}$.

    Since $H_k\geq F_k$ and $\chi_k$ is smooth, we immediately obtain $\mathrm{(I)}_k\leq C_3 H_k\omega_{\mathrm{std}}$ locally around $x$.
    Regarding $\mathrm{(II)}_k$ we apply the Cauchy-Schwarz inequality twice to obtain
    \begin{multline*}
        \mathrm{(II)}_k=\sum_j d\lvert f_{k,j}\rvert^2 \wedge \dc \chi_k+ \de \chi_k \wedge \dc \lvert f_{k,j} \rvert^2 \leq 2\sum_j \left( \left(\sum_{s=1}^n \left\lvert\frac{\partial \lvert f_{k,j} \rvert^2}{\partial z_s}\right\rvert^2 \right)^{1/2}\left(\sum_{s=1}^n\left\lvert \frac{\partial \chi_k}{\partial z_s}\right\rvert^2 \right)^{1/2}\right) \om_{\mathrm{std}}\\
        \leq C_4 \sum_j\left( \lvert f_{k,j} \rvert \left(\sum_{s=1}^n \left\lvert \frac{\partial f_{k,j}}{\partial z_s} \right\rvert^2\right)^{1/2}\right)\om_{\mathrm{std}}\leq C_5 \left(\sum_j \lvert f_{k,j} \rvert^2\right)^{1/2}\left(\sum_j\sum_{s=1}^n \left\lvert \frac{\partial f_{k,j}}{\partial z_s} \right\rvert^2\right)^{1/2}\om_{\mathrm{std}}\leq C_5 H_k \, \om_{\mathrm{std}}
    \end{multline*}
    around $x$, where we clearly again used that $\chi_k$ is smooth.
    Finally,
    $$
    \mathrm{(III)_k}=\sum_j \chi_k \ddc \lvert f_{k,j} \rvert^2\leq C_6 \sum_j i \partial f_{k,j} \wedge \overline{ \partial f_{k,j}} \leq C_7 \left(\sum_j\sum_{s=1}^n\left\lvert\frac{\partial f_{k,j}}{\partial z_s}\right\rvert^2 \right)\om_{\mathrm{std}}\leq C_7 H_k\, \om_{\mathrm{std}},
    $$
    which as said above gives the claim~\eqref{eqn:Claim} and concludes the proof.
\end{proof}
Note that the slightly stronger inequality \eqref{eqn:Stronger} holds.

\section{Fujita Approximations}\label{sec:4}
In the whole section, $L\to X$ denotes a big line bundle on a smooth projective variety $X$. 

\smallskip

We can give the following definition of (Semiample) Fujita Approximations.

\begin{defi}
    The data $(p_m:Y_m\to X, L_m,E_m)$ is said to be a \emph{Fujita Approximation of} $L\to X$ if $p_m:Y_m\to X$ are birational morphisms from smooth projective varieties $Y_m$ such that
    \begin{itemize}
        \item[i)] $p_m^*L=L_m+E_m$ for $E_m$ effective $\Q$-divisors and $L_m$ ample $\Q$-line bundles, for any $m\in\N$;
        \item[ii)] $ L_m^n\longrightarrow \Vol(L)$ as $m\to +\infty$.
    \end{itemize}
    In the case $L_m$ are only semiample $\Q$-line bundles, i.e. suitable multiples of $L_m$ are globally generated, we say that $(p_m:Y_m\to X, L_m, E_m)$ is a \emph{Semiample} Fujita Approximation of $L\to X$.
\end{defi}
Since semiample classes $c_1(L_m)$ are in particular nef classes, the top intersection product $L_m^n$ still captures the volume of the associated line bundles for Semiample Fujita Approximations. Moreover, in the semiample setting there is also no loss of generality in assuming that the morphisms $p_m$ are given as a composition of blow-ups along smooth centers.

Fujita, in his pioneering work~\cite{Fuj94}, showed that any big line bundle admits a Semiample Fujita Approximation given by blowing-up the base ideals associated to the line bundle. Fujita Approximations can then be obtained by a perturbation argument: we will also see this during the proof of Theorem~\ref{thmB} below.

\subsection{Special Fujita Approximations}
Given a big line bundle $L\to X$, the first Riemann-Roch coefficient of $(X,L)$ is defined as
$$
\tau_1(X,L)=\langle L^{n-1}\rangle \cdot  K_X ,
$$
using the same notation as in~\cite[Def.~4.1]{Li21}. 
The discussion in~\cite{Li21} naturally leads to the following definition.
\begin{defi}
    A (Semiample) Fujita Approximation $(p_m:Y_m\to X, L_m,E_m)$ of $L\to X$ is said to be \emph{Special} if
    $$
    \tau_1\left(Y_m,L_m\right)\longrightarrow \tau_1(X,L).
    $$
    as $m\to +\infty$.
\end{defi}
By Projection Formula $\tau_1(Z,q^*F)= \tau_1(Y,F)$ if $q:Z\to Y$ is a birational morphism between smooth projective varieties and $F$ is semiample. Thus, without loss of generality, the morphisms $p_m:Y_m\to X$ of any Special Semiample Fujita Approximation $(p_m:Y_m\to X, L_m,E_m)$ of $L\to X$ can be assumed to be given as compositions of blowups along smooth centers.

\smallskip

The following result has already been partially observed in~\cite[Lem.~4.10]{Li21}.
\begin{lem}
\label{lem:Special_H}
    A Semiample Fujita Approximation $(p_m:Y_m\to X, L_m,E_m)$ of $L\to X$ is Special if and only if \begin{equation*}
        L_m^{n-1}\cdot K_{Y_m/X}\longrightarrow 0 
    \end{equation*}
    as $m\to +\infty$.
\end{lem}
\begin{proof}
    As $L_m\to Y_m$ are semiample, $L_m^{n-1}=\langle L_m^{n-1}\rangle$. Thus
    $$
    L_m^{n-1}\cdot K_{Y_m/X}= \tau_1(Y_m,L_m)-L_m^{n-1}\cdot p_m^*K_X,
    $$
    and to conclude the proof it remains to show that
    $
    L_m^{n-1}\cdot p_m^*K_X\rightarrow \tau_1(X,L)
    $
    as $m\to +\infty$. Moreover, writing $K_X=A-B$ as a difference of ample divisors, it is enough to check that $L_m^{n-1}\cdot p_m^*A\rightarrow \langle L^{n-1}\rangle\cdot A$ as $m\to +\infty$ for $A$ ample divisor.
    
    By definition we immediately have $\limsup_{m\to +\infty}L_m^{n-1}\cdot p_m^* A\leq \langle L^{n-1}\rangle\cdot A$ (see subsection~\ref{ssec:Pos_Int_Prod}). Vice versa, as $p_m^*(L-tA)\geq L_m-tp_m^*A$ for any $m\in \N$ and any $t>0$, we obtain
    \begin{align*}
        \Vol(L)-\Vol(L-tA)\leq \underbrace{\Vol(L)-\Vol(L_m)}_{=:\varepsilon_m} + \Vol(L_m)- \Vol(L_m - tp_m^*A)
        \leq\varepsilon_m +nt \, L_m^{n-1}\cdot p_m^*A
    \end{align*}
    where the last inequality is the content of~\cite[Lem.~4.2]{BDPP13} as $L_m, tp_m^*A$ are nef. Since by hypothesis $\varepsilon_m \to 0$ as $m\to +\infty$, we can consider a sequence $\{t_m\}_m$ converging to $0$ such that $\varepsilon_m/t_m\to 0$, to get
    $$
    \liminf_{m\to +\infty}\frac{\Vol(L)-\Vol(L-t_mA)}{nt_m}\leq \liminf_{m\to+\infty} L_m^{n-1}\cdot p_m^*A.
    $$
    Thus the desired inequality $\langle L^{n-1} \rangle\cdot A\leq \liminf_{m\to +\infty}L_m^{n-1}\cdot p_m^*A $ follows from~\cite[Thm.~A]{BFJ09}, concluding the proof.
\end{proof}

\section{Existence of Special Fujita Approximations}\label{sec:5}
In this section we will prove Theorem~\ref{thmB}, solving in particular~\cite[Conj.~4.7]{Li21} as stated in the Introduction. We will also see that the considered morphisms $p_m:Y_m\to X$ in Theorem~\ref{thmB} can be taken to be compositions of blow-ups along ($\G$-invariant) smooth centers such that the exceptional loci are simple normal crossing. Indeed, they will be constructed as log resolutions of the base ideals associated with $L$.
\medskip

    The author thanks S. Boucksom for suggesting the following algebraic proof of Theorem~\ref{thmB} after the first arXiv version of the manuscript.
    \begin{proof}[Proof of Theorem~\ref{thmB}] Letting $p:Y\to X$ be a fixed birational morphism, we have $\mathbf{B}_+(p^*L)=p^{-1}\left(\mathbf{B}_+(L)\right)\cup \Exc(p)$ by~\cite[Prop.~2.3]{BBP13}. In particular, $\langle (p^*L)^{n-1}\rangle\cdot K_{Y/X}=0$ by Theorem~\ref{thm:Orthogonality} as $\mathrm{Supp}( K_{Y/X})\subset \Exc(p)$. Thus, without loss of generality, unless considering $p^*L$ for a suitable birational morphism $p:Y\to X$, we can and will assume that $\mathbf{B}_+(L)=\mathrm{Supp}(E)$ for an effective $\Q$-divisor $E$ such that $A:=L-E$ is an ample $\Q$-line bundle.
    
    \textbf{Step 1: Enough to prove the existence of Special Semiample Fujita Approximations.}    
    We want to perturb a given Special Semiample Fujita Approximation $(p_m:Y_m\to X, L_m,E_m)$ of $L\to X$ to produce a Special Fujita Approximation. Without loss of generality, we can and will suppose that $p_m$ is given as a composition of blowups along smooth centers. In particular, there are $p_m$-exceptional effective $\Q$-divisors $F_m$ such that $p_m^*A-F_m$ are $\Q$-ample classes.
    A simple calculation gives
    \begin{align*}
        p_m^*L&= a(L_m+E_m)+ (1-a)\left(p_m^*A +p_m^*E\right)\\
        &= a L_m + (1-a)(p_m^*A-F_m) + aE_m+ (1-a)p^*_m E + (1-a) F_m
    \end{align*}
    for any $a\in \Q \cap (0,1)$. Set $L_{m,a}:=aL_m+(1-a)(p_m^*A-F_m)$ and $E_{m,a}:= aE_m+ (1-a)p^*_m E + (1-a) F_m$, noting that $L_{m,a}$ are $\Q$-ample line bundles while $E_{m,a}$ are effective $\Q$-divisors.
    By continuity we can then choose $a_m\in(0,1)\cap \Q$ so that
    \begin{gather*}
        L_{m,a_m}^n\geq L_m^n-\frac{1}{m},\quad \quad
        L_{m,a_m}^{n-1}\cdot K_{Y_m/X}\leq L_m^{n-1}\cdot K_{Y_m/X}+\frac{1}{m}
    \end{gather*}
    for any $m\in\N$. It follows from Lemma~\ref{lem:Special_H} that $(p_m:Y_m\to X, L_{m,a_m}, E_{m,a_m})$ is a Special Fujita Approximation of $L\to X$.  

    \textbf{Step 2: Comparing base ideals and asymptotic multiplier ideal sheaves.}
    In this step we want to prove that
    \begin{equation}
        \label{eqn:AB2}
        \mathfrak{b}_m \cdot \mathfrak{c}^2\subset \mathfrak{a}_m
    \end{equation}
    for any $m$ large, where $\{\mathfrak{a}_m\}_{m\in\N}, \{\mathfrak{b}_m\}_{m\in\N}$ and $\mathfrak{c}$ are suitable ideal sheaves we are going to introduce.
    
    We start by letting $\mathfrak{a}_m:=\mathfrak{b}\left(\left| mL\right|\right)$ be the sequence of base ideals of $L$, and we set $\mathfrak{b}_m:=\mathcal{J}\big(\mathfrak{a}_{\bullet}^m\big)=\mathcal{J}\big(m\cdot \lVert L \rVert\big)=\mathcal{J}\big(\lVert mL \rVert\big)$ for the multiplier ideal sheaves of the graded families of ideals $\mathfrak{a}_{\bullet}^m$ (see subsection~\ref{ssec:Ideals}). 

    Next, we claim that there exists $m_0\in\N$ such that $\mathcal{F}_m:=O_X(mL+m_0A)\otimes \mathfrak{b}_m$ is globally generated. Indeed, picking $r\in\N$ such that $rA$ is a very ample divisor, the claim would follow by Castelnuovo-Mumford regularity~\cite[Thm.~1.8.5]{Laz04} if we proved that $H^i\left(X,\mathcal{F}_m\otimes O_X(-irA)\right)=0$ for any $i=1,\dots,n$. But, writing $\mathcal{F}_m\otimes O_X(-irA)=O_X(K_X+mL+P_i)\otimes \mathfrak{b}_m$ for $P_i=(m_0-ir)A-K_X$ and choosing  $m_0$ so that $P_i$ is an ample divisor for any $i=1,\dots,n$, the desired vanishings are consequences of the asymptotic Nadel vanishing Theorem~\cite[Thm.~11.2.12]{Laz04II}. We can clearly also suppose that $m_0$ is such that $m_0E, m_0A$ are divisors.

    As $mL+m_0A=(m+m_0)L-m_0E$, from the global generation of $\mathcal{F}_m$ we deduce that
    \begin{equation}
        \label{eqn:AB1}
        \mathcal{O}_X(-m_0E)\cdot \mathfrak{b}_m\subset \mathfrak{a}_{m+m_0} 
    \end{equation}
    for any $m\in\N$. Set $\mathfrak{c}:=\mathcal{O}_X(-m_0E)$ and also recall that
    \begin{equation}
        \label{eqn:AB3}
        \mathfrak{a}_m\subset \mathcal{J}\big(\mathfrak{a}_m\big)\subset \mathfrak{b}_m
    \end{equation}
    (see~\cite[Lem.~B.4]{BBJ15}). Thus, combining~\eqref{eqn:AB1},~\eqref{eqn:AB3} with the subadditivity of multiplier ideal sheaves (\cite{DEL00}, see also~\cite[Thm.~11.2.3]{Laz04II}), it follows that
    \begin{align*}
        \mathfrak{b}_m \cdot \mathfrak{c}^2\subset \mathfrak{b}_{2m_0} \cdot \mathfrak{b}_{m-2m_0}\cdot  \mathfrak{c}^2 \subset \mathfrak{b}_{2m_0} \cdot \mathfrak{a}_{m-m_0} \cdot \mathfrak{c}\subset  \mathfrak{b}_{2m_0} \cdot \mathfrak{b}_{m-m_0} \cdot \mathfrak{c} \subset \mathfrak{b}_{2m_0}\cdot \mathfrak{a}_{m}\subset \mathfrak{a}_m
    \end{align*}
    for any $m\geq 2m_0$.

    \textbf{Step 3: The chosen Special Semiample Fujita Approximation.}
    Consider the Semiample Fujita Approximation $(p_m:Y_m\to X, L_m,E_m)$ given by taking log resolutions of the base ideals $\mathfrak{a}_m:=\mathfrak{b}\left(\left|mL\right|\right)$. By the universal property of the normalized blowup, we have the commutative diagram
    \[\begin{tikzcd}
    Y_m \arrow{rr}{q_m} \arrow[swap]{dr}{p_m} & & Z_m \arrow{dl}{\pi_m} \\
    & X &
    \end{tikzcd}
    \]
    where $\pi_m:Z_m\to X$ is the normalized blow-up along $\mathfrak{a}_m$.
    From~\eqref{eqn:AB2},~\eqref{eqn:AB3} and with the same notation, we know that
    $$
    \ord_F \mathfrak{a}_m - \ord_F \mathcal{J}\left(\mathfrak{a}_m\right)\leq \ord_F \mathfrak{a}_m - \ord_F \mathfrak{b}_m\leq 2 \ord_F \mathfrak{c}
    $$
    for any $m\geq 2 m_0$ and for any prime divisor $F$ over $X$ (i.e. $F\subset Y$ prime divisor on a normal projective variety $Y$ that dominates $X$ through a birational morphism $p:Y\to X$). Thus Proposition~\ref{thm:Discrepancies_Bouck} yields
    \begin{equation}
        \label{eqn:Mat1}
        K_{Z_m/X}\leq 4n m_0\pi_m^* E
    \end{equation}
    for any $m\geq 2 m_0$. Next, by construction, we also have $ E_m=q_m^*\hat{E}_m$ for an effective $\Q$-divisor $\hat{E}_m$ such that $ \mathcal{O}_{Z_m}\cdot \pi_m^{-1}\mathfrak{a}_m=\mathcal{O}_{Z_m}(-m\hat{E}_m) $. We deduce that
    $$
    L_m= p_m^*L - E_m = q_m^*\left(\pi_m^* L -\hat{E}_m\right)=q_m^*M_m,
    $$
    where $M_m$ is a nef $\Q$-line bundle. Thus, since $q_{m,*}K_{Y_m/X}=K_{Z_m/X}$ as Weil divisors, from \eqref{eqn:Mat1} and the Projection Formula it follows that 
    $$
    \limsup_{m\to+\infty}L_m^{n-1}\cdot K_{Y_m/X}=\limsup_{m\to+\infty}M_m^{n-1}\cdot K_{Z_m/X}\leq 4nm_0\, \limsup_{m\to+\infty} M_m^{n-1}\cdot \pi_m^*E= 4nm_0 \langle L^{n-1}\rangle\cdot E=0,
    $$
    where the last convergence is deduced from pulling-back to $Y_m$ as $(p_m:Y_m\to X,L_m, E_m)$ is a Semiample Fujita Approximation (cf. the proof of Lemma \ref{lem:Special_H}), while the last equality is given by Theorem~\ref{thm:Orthogonality} as $\mathbf{B}_+(L)=\mathrm{Supp}(E)$. Hence, the chosen Semiample Fujita Approximation is Special thanks to Lemma~\ref{lem:Special_H}.
        
    \textbf{Step 4: the $\G$-equivariant case.}
    Let us now assume that $\G=K_\C$ is a reductive Lie group, $K$ compact Lie group, and that $L$ is $\G$-linearized. Unless considering $p^*L$ for a suitable $\G$-equivariant morphism $ p:Y\to X$, we can assume that $\mathbf{B}_+(L)=\mathrm{Supp}(E)$ for a $\G$-invariant $\Q$-divisor $E$ such that $A=L-E$ is an ample $\G$-linearized $\Q$-line bundle. Moreover, as the base ideals $\mathfrak{a}_m=\mathfrak{b}(|mL|)$ are $\G$-invariant, we can consider $\G$-equivariant log resolutions $p_m:Y_m\to X$ of $\mathfrak{a}_m$ to obtain a Special Semiample Fujita Approximation $(p_m:Y_m\to X, L_m, E_m)$ such that the semiample $\Q$-line bundles $L_m$ are $\G$-linearized and the effective $\Q$-divisors $E_m$ are $\G$-invariant.
    
    Finally it remains to show that we can perform the perturbation argument of the beginning of the proof in a $\G$-equivariant way. Since $p_m$ can be assumed to be given by a finite sequence of blowups along $\G$-invariant smooth centers, there exist $\G$-invariant $p_m$-exceptional effective $\Q$-divisors $F_m$ such that $p_m^*A-F_m$ are $\G$-linearized ample $\Q$-divisors. In particular, the perturbed $\Q$-line bundles $L_{m,a}=aL_m+(1-a)(p_m^*A-F_m)$ of Step 1 are still $\G$-linearized, and similarly the perturbed effective $\Q$-divisors $E_{m,a}=aE_m + (1-a)p_m^*E+(1-a) F_m$ are still $\G$-invariant. Hence, the same argument of Step 1 applies and concludes the proof.
    \end{proof}

    \section{Solution to the Yau-Tian-Donaldson Conjecture}\label{sec:6}

    As explained in the Introduction, Theorem~\ref{thmB} is known to imply solutions to the Boucksom-Jonsson Regularization Conjecture, i.e. Corollary~\ref{corA}, and to the Yau-Tian-Donaldson Conjecture, i.e. Theorem~\ref{thmA}. These results, Corollary~\ref{corB} and Theorem~\ref{thmC} are the contents of this section.
    \smallskip

    \begin{proof}[Proof of Theorem~\ref{thmA} and Corollary~\ref{corA}]
        As explained in the Introduction and recalled in Theorem~\ref{thm:YTD_F}, Corollary~\ref{corA} implies Theorem~\ref{thmA}. Moreover, by what is recalled in subsection~\ref{ssec:RC}, it is enough to prove Conjecture~\ref{conj:RC} to get Corollary~\ref{corA}. Let then $\G\subset \Aut^\circ(X,L)$ be a reductive Lie group and let $(\cX,\cL)$ be a smooth, dominating and simple normal crossing $\G$-equivariant big model. 
        We denote by $p:\cX\to X\times \P^1$ the $\C^*, \G$-equivariant dominating morphism to the trivial test configuration $X\times \P^1$ for $X$, while we denote by $\pi:\cX\to \P^1$ the $\C^*$-equivariant map given by $\pi=p_2\circ p$. Observe that, since $\cX_{|\pi^{-1}(\P^1\setminus\{0\})}\overset{p}{\simeq} X\times (\P^1\setminus\{0\})$, we can assume that $p$ is given by a sequence of blowups along $\C^*,\G$-invariant smooth centers whose images in $X\times \P^1$ are contained in the central fiber, unless considering an equivalent test configuration. In particular $\cX$ is endowed with a $\C^*,\G$-linearized ample line bundle $\mathcal{A}=p^*p_1^*A-\sum_j a_j \mathcal{F}_j+a\cX_0$ for a suitable coefficient $a$ and $\C^*,\G$-invariant $p$-exceptional divisors $\mathcal{F}_j$ with support in $\cX_0$. We can then apply Theorem~\ref{thmB} to get a Special Ample Fujita Approximation that is $\G$-equivariant. Namely, there exist $\G$-equivariant morphisms $p_m:\cY_m\to \cX$, ample $\G$-linearized $\Q$-line bundles $\cL_m$ and $\G$-invariant effective $\Q$-divisors $\mathcal{E}_m$ such that
        $$
        p_m^*\cL=\cL_m+\mathcal{E}_m\quad \quad \cL_m^{n+1}\longrightarrow \Vol(\cL) \quad \quad \tau_1(\cY_m,\cL_m)\longrightarrow \tau_1(\cX,\cL)
        $$
        as $m\to +\infty$. It remains to check that a subsequence of $\{(\cY_m,\cL_m)\}_m$ is given by test configurations. Following the proof of Theorem~\ref{thmB} we observe that the morphisms $p_m$ can be taken to be $\C^*,\G$-equivariant log resolution of $\C^*,\G$-invariant ideals $\mathcal{I}_m$. 
        In particular the $\Q$-line bundles $\cL_m$ are $\C^*, \G$-linearized and the $\Q$-divisors $\mathcal{E}_m$ are $\C^*,\G$-invariant. Therefore to prove that $(\cY_m,\cL_m)$ are ample test configurations for $(X,L)$ it is enough to show that the $\Q$-divisors $\cE_m$ have support contained in the central fibers, which follows once the support of $\mathcal{I}_m$ is contained in the central fiber $\cX_0$. But by construction $V(\mathcal{I}_m)\subset \mathbf{B}_+(\mathcal{L})=\mathbf{B}(m\cL-\cA)$ where the equality is valid for any $m\gg 1 $ large enough. 
        Unless replacing $\cL$ by $\cL+c\cX_0$ for a large $c\gg 0$ we obtain that $\mathbf{B}_+(\cL)\subset \cX_0$ as $(\cX,\cL)_{|\pi^{-1}\left(\P^1\setminus \{0\}\right)}\simeq \left(X\times \P^1\setminus \{0\}, p_1^*L\right)$ and $L$ is ample. Thus, the aforementioned ideals $\mathcal{I}_m$ have support in $\cX_0$. As said above, this finishes the proof.
    \end{proof}

    We can now prove Corollary~\ref{corB}.
    \begin{proof}[Proof of Corollary~\ref{corB}]
    We already know that $(iv)\Rightarrow (iii)\Rightarrow(i)$ and Corollary~\ref{corA} gives $(i)\Rightarrow(iv)$.
    By definition we also have $(v)\Rightarrow (ii)$. Moreover Theorem~\ref{thmB} naturally leads to a non-equivariant version of Corollary~\ref{corA}. The latter, together with~\cite[Prop.~6.3]{Li22b} (see also~\cite[Prop.~7.14]{BJ26}), leads to $(ii)\Rightarrow (v)$ by definition of the \emph{Non-Archimedean Mabuchi functional} (see also~\cite[Sec.~5.5]{BJ23II}).
    Theorem~\ref{thmA} and~\cite[Thm.~A]{BJ26} then gives that $(i)\Leftrightarrow (v)\Leftrightarrow (vi)$.
    Finally, in the case of discrete automorphism group~\cite[Thm.~1.1]{DZ26} and again Theorem~\ref{thmA} yield $(vii)\Leftrightarrow (i)$, concluding the proof.
    \end{proof}

    We conclude the paper by providing a proof of Theorem~\ref{thmC}.
    \begin{proof}[Proof of Theorem~\ref{thmC}]
        As said in the Introduction it is enough to prove that $T_\C$-uniform $(v,w)$-weighted relative $K$-stability implies the analogous notion for models. Let then $(\cX,\cL)$ be a smooth, dominating and simple normal crossing $T_\C$-equivariant big model for $(X,L)$. As seen in the proof of Theorem~\ref{thmA} we can provide a Special Fujita Approximation of $(\cX,\cL)$ given by $T_\C$-equivariant morphisms $p_m:\cY_m\to \cX$, ample $T_\C$-linearized $\Q$-line bundles $\cL_m$ and $T_\C$-invariant effective $\Q$-divisors $\mathcal{E}_m$. Moreover $(\cY_m,\cL_m)$ can be constructed as a sequence of $T_\C$-equivariant smooth ample test configurations of $(X,L)$. To conclude the proof we then need to show that the $(v,w)$-\emph{weighted relative Non-Archimedean Mabuchi functional} is continuous with respect to this approximation. More precisely, we know that the associated Non-Archimedean metrics are strongly continuous in the Boucksom-Jonsson's sense by~\cite[Lem.~4.5]{Li21}. Together with \cite[Lem.~3.7,~Lem.~3.10,~Cor.~3.13]{HL25b} this implies that it is enough to prove that 
        $$
        \mathrm{H}^{\NA}_v(\cY_m,\cL_m)\longrightarrow \mathrm{H}^{\NA}_v(\cX,\cL)
        $$
        as $m\to +\infty$, where $\mathrm{H}^{\NA}_v$ is the $v$-weighted Non-Archimedean entropy introduced in~\cite[eqn.~(3.27)]{HL25b}. Following the notations in~\cite{HL25b}, unless considering a base change of the original big test configuration, by~\cite[Prop.~4.1]{HL25b} we have
        $$
        \mathrm{H}^{\NA}_v(\cY_m,\cL_m)=\big(\cL_m^n\cdot K_{\cY_m/X_{\P^1}}^{\log}\big)_v, \,\,\,\,\mathrm{H}^{\NA}_v(\cX,\cL)=\langle\cL^n\rangle_v\cdot K_{\cX/X_{\P^1}}:=\sup_{k\to +\infty}\big(\tilde{\cL}_k^n\cdot \rho_k^*K_{\cX/X_{\P^1}}\big)_v
        $$
        where the intersection products are twisted by the presence of the weight $v$ and where $(\tilde{\cY}_k,\tilde{\cL}_k)$ are the test configurations obtained by blowing-up the base ideals associated to $k\cL$. We refer to~\cite{HL25b} for the precise definition of the product $(\cdots)_v$. Here we just recall that $(\cF^n, \cdot)_v$ acts as integration against the Non-Archimedean $v$-weighted Monge-Ampère measure $\mathrm{MA}^{\NA}_v(\phi_{\cF})$ associated to an ample test configuration $(\cZ,\cF)$. For our purposes we deduce that $(\cdots)_v$ is linear in the last entry and that
        \begin{equation}
            \label{eqn:FFF}
            \big(\cL_m^n\cdot p_m^*K_{\cX/X_{\P^1}}\big)_v\longrightarrow \langle\cL^n\rangle_v\cdot K_{\cX/X_{\P^1}} 
        \end{equation}
        as $m\to +\infty$. Indeed, as said before the Non-Archimedean metrics $\phi_{\cL_m}$ strongly converges to the Non-Archimedean metric associated to $\left(\cX,\cL\right)$ and the Non-Archimedean $v$-weighted Monge-Ampère operator is strongly continuous (cf.~\cite[Thm.~3.20]{HL25b}). On the other hand by~\cite[Rem.~4.2]{HL25b} there exists a constant $C>0$ such that
        \begin{equation}
            \label{eqn:FFFF}
            (\cL_m^n\cdot K_{\cY_m/\cX})_v\leq C (\cL_m^n\cdot K_{\cY_m/\cX})\longrightarrow 0.
        \end{equation}
        Combining~\eqref{eqn:FFF},~\eqref{eqn:FFFF} and $K_{\cY_m/X_{\P^1}}^{\log}\leq K_{\cY_m/X_{\P^1}}= K_{\cY_m/\cX}+p_m^*K_{\cX/X_{\P^1}}$, we deduce that
        \begin{equation*}
            \mathrm{H}^{\NA}_v(\cY_m,\cL_m) \leq (\cL_m^n\cdot K_{\cY_m/\cX})_v+ (\cL_m^n\cdot p_m^*K_{\cX/X_{\P^1}})_v\longrightarrow \langle\cL^n\rangle_v\cdot K_{\cX/X_{\P^1}} =\mathrm{H}^{\NA}_v(\cX,\cL)
        \end{equation*}
        which concludes the proof, as the lower-semicontinuity $\liminf_{m\to +\infty}\mathrm{H}^{\NA}_v(\cY_m,\cL_m)\geq H_v^{\NA}(\cX,\cL)$ follows from the weak convergence of the associated Non-Archimedean $v$-weighted Monge-Ampère measures, the lower-semicontinuity of the log-discrepancy function and the formula \cite[eqn.~(3.27)]{HL25b}.
    \end{proof}
    
    \appendix

    \section{On the entropy approximation conjecture\\
    {\small (by S. Boucksom)}} 
    \label{appendix}
    
    The \emph{entropy approximation conjecture} (aka \emph{regularization conjecture} in the main body of the paper), which originates in~\cite[Conj.~2.5]{BJ18} (see also~\cite[Conj.~4.4]{Li21}, \cite[Conj.~5.12]{BoJ25a}) states that the non-Archimedean entropy functional is the greatest lower semicontinuous (lsc) extension of its restriction to the subspace of Fubini--Study metrics. It was inspired by its complex analytic counterpart, established in~\cite{BDL17}. 
    
    One of the main results of the present paper (Corollary~\ref{corA}) settles this conjecture over the field $\C$ equipped with the trivial valuation. 

    The primary purpose of this appendix is to provide a simple, purely algebro-geometric proof of the discrepancy estimate for Rees valuations (Proposition~\ref{thm:Discrepancies_Bouck}), based on Nadel vanishing and valid over any algebraically closed field $k$ of characteristic $0$. We then apply this to establish the entropy approximation conjecture over the field of Laurent series $k\lau{t}$, a case of relevance for instance in the study of $K$-stability with respect to arcs~\cite{DR24}. 

    \subsection{The discrepancy estimate for Rees valuations} 
    In what follows, $\cX$ denotes a smooth algebraic variety over an algebraically closed field $k$ of characteristic $0$. Slightly generalizing Proposition~\ref{thm:Discrepancies_Bouck} (in which $k=\C$), we will show: 

    \begin{prop}\label{prop:key}
    Pick an ideal\footnote{By this we mean a coherent ideal sheaf.} $\fa\subset\cO_\cX$, and a Rees valuation $v$ of $\fa$, \ie $v=\ord_E$ for a prime divisor $E$ in the support of the exceptional divisor of the normalized blowup $\rho\colon \tcX\to \cX$ of $\fa$, with discrepancy $a(v)=v(K_{\tcX/\cX})$. Then 
    $$
    a(v)\le(\dim\cX)\left(v(\fa)-v(\cJ(\fa))\right), 
    $$
    with $\cJ(\fa)\subset\cO_\cX$ the multiplier ideal of $\fa$.
    \end{prop}

    \begin{proof}
    Let $D$ be the exceptional divisor of the normalized blowup, \ie the effective Cartier divisor on $\tcX$ such that $\fa\cdot\cO_{\tcX}=\cO_{\tcX}(-D)$. Since $D$ is the pullback of the exceptional divisor of the blowup of $\fa$ under the (finite) normalization morphism, $-D$ is relatively ample and globally generated over $\cX$. 

    Pick a log resolution $\pi\colon \cX'\to \cX$ of $\fa$ that factors through a projective birational morphism $\mu\colon \cX'\to \tcX$. Setting $m:=\dim\cX$, we then have by definition of multiplier ideals
    $$
    \cJ(\fa^m)=\pi_\star\cF,\quad\cF:=\cO_{\cX'}\left(K_{\cX'/\cX}-m\mu^\star D\right). 
    $$
    Given a prime divisor $E$ in the support of $D$, with strict transform $E'\subset \cX'$, we claim that $\cF$ is relatively globally generated over $\cX$ at the generic point of $E'$. Since $\cJ(\fa^m)\subset\cJ(\fa)^m$ by the subadditivity property of multiplier ideals, this will prove 
    $$
    m\ord_E\left(\cJ(\fa)\right)\le \ord_E\left(\cJ(\fa^m)\right)=-\ord_{E'}\left(K_{\cX'/\cX}-m\mu^\star D\right)=-a(\ord_E)+m\ord_E(\fa), 
    $$
    which yields the desired estimate. 

    To prove the claim, pick a general (closed) point $p\in E$. Since $-D$ is relatively ample and basepoint free over $\cX$, the sheaf of relative sections $\rho_\star\cO_{\tcX}(-D)$ defines a finite morphism to a projective space over $\cX$. We can thus find finitely many relative sections $s_i\in\rho_\star \cO_{\tcX}(-D)$ with an isolated common zero at $p$, so that the ideal $\fb\subset\cO_{\tcX}$ they locally generate is nontrivial at $p$, and trivial at any other nearby point of $\tcX$. 

    Since $\tcX$ is normal, $\mu\colon\cX'\to\tcX$ is an isomorphism over the generic point of $E$, and hence over a neighborhood of $p$. The ideal $\fb':=\fb\cdot\cO_{\cX'}$ locally generated by the relative sections $\mu^\star s_i\in\pi_\star\cO_{\cX'}(-\mu^\star D)$ is thus nontrivial at the preimage $p'\in \cX'$ of $p$, and trivial at any other nearby point of $\cX'$. Using basic properties of multiplier ideals, it follows that the subscheme $\cZ\subset\cX'$ cut out by $\cJ(\fb'^m)$ contains $p'$ as an isolated point. By (relative) Nadel vanishing, we further have $R^1\pi_\star\left(\cF\otimes\cJ(\fb'^m)\right)=0$,
    so that the restriction map $\pi_\star\cF\to\pi_\star(\cF\otimes\cO_\cZ)$ is surjective. We infer that $\cF$ is relatively globally generated over $\cX$ at $p'$, which proves the claim, and concludes the proof.
    \end{proof}

    \begin{rmk}\label{rmk:model}
    The above argument, which relies solely on the fundamental properties of multiplier ideals, applies without change when $\cX$ is a regular flat projective scheme over $k\cro{t}$ and the ideal $\fa\subset\cO_\cX$ is vertical, \ie cosupported in the central fiber, see~\cite[App.~B]{BFJ16}. 
    \end{rmk}

    \subsection{The entropy approximation conjecture}
    Denoting as above by $k$ an algebraically closed field of characteristic $0$, we next consider a smooth, $n$-dimensional projective variety $X$ over the field $K:=k\lau{t}$ of formal Laurent series, together with an ample line bundle $L$ on $X$. We view $K$ as a complete non-Archimedean field, and denote by $X^\an$ the Berkovich analytification of $X$, viewed as a compact Hausdorff topological space whose points are semivaluations $v$ on $X$.

    We will use freely the language of non-Archimedean pluripotential theory, using~\cite{BFJ16,BoJ17,BJ22,BoJ25a} as references. The space $\cE^1=\cE^1(L^\an)$ of \emph{psh metrics of finite energy} $\phi$ on $L^\an$ is equipped with its strong topology, and contains the space $\cH$ of \emph{Fubini--Study metrics} as a dense subset. The (non-Archimedean) \emph{Monge--Amp\`ere operator} 
    $$
    \phi\mapsto\MA(\phi):=\vol(L)^{-1}(\ddc\phi)^n
    $$
    is continuous on $\cE^1$, with values in the space of probability measures on $X^\mathrm{na}$. The \emph{non-Archimedean entropy} of $\phi\in\cE^1$ is defined as 
    \begin{equation}\label{equ:H}
    \hh(\phi):=\int_{X^\an}(A_X-\p)\MA(\phi).
    \end{equation}
    Here $A_X$ denotes the \emph{Temkin metric} on $K_X^\an$ (using additive notation for metrics), and $\p$ is a choice of reference model metric on $K_X^\an$, see~\cite[\S 5.2.2]{BoJ25a}. The singular metric $A_X$ is merely lsc, and can be written as the upper envelope of the family of model metrics 
    $$
    \phi_\cX:=\phi_{K^{\log}_{\cX/S}}
    $$ 
    on $K_X^\an$ associated to each snc model $\cX\to S:=\Spec k\cro{t}$ of $X$. 
    Further, $A_\cX:=A_X-\phi_\cX$ coincides with the log discrepancy function $A_\cX\colon X^{\mathrm{na}}\to[0,+\infty]$ of the log smooth pair $(\cX,\cX_{0,\redu})$, which vanishes precisely on the associated dual complex $\D_\cX\subset X^{\mathrm{na}}$, see for instance~\cite[\S 5]{BoJ17}. 

    The non-Archimedean entropy functional $\hh\colon\cE^1\to\R\cup\{+\infty\}$ is merely lsc. Adapting the arguments leading to Corollary~\ref{corA}, we next establish the entropy approximation conjecture in the present setup. 

    \begin{thm}\label{thm:ent}
    The functional $\hh\colon\cE^1\to\R\cup\{+\infty\}$ is the greatest lsc extension from its restriction to $\cH$; equivalently, any $\phi\in\cE^1$ can be written as the limit of a sequence $(\phi_m)$ in $\cH$ such that $\hh(\phi_m)\to\hh(\phi)$. 
    \end{thm}

    \begin{rmk} As the proof will show, the approximating sequence is obtained in a sufficiently canonical way to ensure that invariance under a given algebraic group action is preserved.
    \end{rmk}

    The non-Archimedean entropy differs from the \emph{non-Archimedean Mabuchi functional} 
    $$
    \mab\colon\cE^1\to\R\cup\{+\infty\}
    $$
    by the $J$-functional, which is finite-valued and continuous on $\cE^1$, see~\cite[\S 5.2]{BoJ25a}. As a result, the conclusion of Theorem~\ref{thm:ent} also applies to $\mab$, which implies:

    \begin{cor}  The non-Archimedean Mabuchi functional $\mab\colon\cE^1\to\R\cup\{+\infty\}$ is bounded below (resp.~coercive) on $\cE^1$ iff it is on $\cH$. 
    \end{cor}
    When $(X,L)$ is obtained by ground field extension from $k$, boundedness from below (resp.~coercivity) of $\mab$ on $\cH$ corresponds to $K$-semistability (resp.~uniform $K$-stability) with respect to arcs, in the sense of~\cite{DR24}. 

    \medskip

    In a first step towards the proof of Theorem~\ref{thm:ent}, arguing exactly as in the proof of~\cite[Thm.~3.10]{BJ23II} and~\cite[Prop.~6.3]{Li22b} (which originates in~\cite[\S 7.4]{BoJ18a}) by projecting the measure $\MA(\phi)$ to a dual complex and solving the Monge--Ampère equation, one reduces to the case where $\phi=\env(\phi_\cL)$ is the psh envelope of a model metric associated to an snc model $(\cX,\cL)$ of $(X,L)$. Theorem~\ref{thm:ent} is now a consequence of the following more precise result: 

    \begin{thm} Let $\phi_\cL$ be the model metric on $L^{\mathrm{na}}$ associated to an snc model $(\cX,\cL)$ of $(X,L)$. Denote by $\phi=\env(\phi_\cL)$ its psh envelope, and by $\phi_m$ the Fubini--Study metric associated to the lattice $H^0(\cX,m\cL)$ of $H^0(X,mL)$, for each $m\gg 1$. Then $\phi_m\to\phi$ uniformly, and $\hh(\phi_m)\to\hh(\phi)$. 
    \end{thm}

    \begin{proof} Any vertical ideal $\fa$ defines a model function $\log|\fa|$ on $X^{\mathrm{na}}$ by setting 
    $$
    \log|\fa|(v):=-v(\fa).
    $$
    Denoting by $\fa_m\subset\cO_\cX$ the relative base ideal of $m\cL$, we then have 
    $$
    \phi_m=\phi_\cL+m^{-1}\log|\fa_m|, 
    $$
    and the first assertion thus amounts to~\cite[Thm.~8.5]{BFJ16}.

    Consider next the irreducible decomposition $\cX_0=\sum_E b_E E$ of the central fiber of $\cX$. Arguing just as in~\cite[Lem.~4.8]{BJ26}, one infers from the orthogonality property that $\MA(\phi)$ is supported in the (finite) set of divisorial valuations $v_E=b_E^{-1}\ord_E\in\D_\cX$ such that $E$ is not contained in the relative augmented base locus $\B_+(\cL/S)$. After perhaps replacing $\cX$ with a higher snc model, we can further assume that the relative augmented base locus $\B_+(\cL/S)$ is the support of an effective vertical $\Q$-divisor $D$ on $\cX$ such that $\cL-D$ is relatively ample over $S$. As a consequence, the corresponding model function $\phi_D\ge 0$ satisfies 
    $$
    \int\phi_D\MA(\phi)=0.
    $$
    By~\eqref{equ:H}, for any $\tau\in\cE^1$ we have 
    $$
    \hh(\tau)=\int (A_\cX+f)\,\MA(\tau),\quad f:=\phi_\cX-\p\in C^0(X^{\mathrm{na}}).
    $$
    Since $\phi_m\to\phi$ uniformly, we have $\MA(\phi_m)\to\MA(\phi)$ weakly, and hence
    $$
    \int f\,\MA(\phi_m)\to\int f\,\MA(\phi)=\hh(\phi),
    $$
    since $\MA(\phi)$ is supported in the dual complex $\D_\cX=\{v\in X^{\mathrm{na}}\mid A_\cX(v)=0\}$. The second point of the theorem is thus equivalent to 
    \begin{equation*}
    \lim_m \int A_\cX\MA(\phi_m)=0.
    \end{equation*}
    We claim
    \begin{equation}\label{equ:bd}
    \int A_\cX\MA(\phi_m)\le C\int \phi_D\MA(\phi_m)
    \end{equation}
    for a uniform constant $C>0$, which will conclude the proof since 
    $$
    \int\phi_D\MA(\phi_m)\to\int\phi_D\MA(\phi)=0.
    $$ 
    As in~\cite[\S 8.1]{BFJ16}, consider the asymptotic multiplier ideal $\fb_m:=\cJ(\fa_\bullet^m)\subset\cO_\cX$, and denote by $\rho_m\colon\cX_m\to\cX$ the normalized blowup of $\fa_m$. By construction, the measure $\MA(\phi_m)$ is supported in the valuations $v_E\in X^\div$ attached to the irreducible components $E$ of the central fiber $\cX_{m,0}$. For each such $E$, Proposition~\ref{prop:key} (and Remark~\ref{rmk:model}) further yields 
    $$
    v_E(K_{\cX_m/\cX})\le(n+1)(v_E(\fa_m)-v_E(\cJ(\fa_m)))\le (n+1)(v_E(\fa_m)-v_E(\fb_m)), 
    $$
    using $\cJ(\fa_m)\subset\fb_m$. Since $A_\cX$ is the log discrepancy function of the pair $(\cX,\cX_{0,\redu})$, we further have 
    $$
    A_\cX(v_E)=v_E\left((K_{\cX_m}+\cX_{m,0,\redu})-\rho_m^\star(K_{\cX}+\cX_{0,\redu})\right)\le v_E(K_{\cX_m/\cX})
    $$
    since $\rho_m^\star\cX_{0,\redu}\ge\cX_{m,0,\redu}$, and we infer
    \begin{equation}\label{equ:bd1}
    \int A_\cX\MA(\phi_m)\le(n+1)\int(\log|\fb_m|-\log|\fa_m|)\MA(\phi_m).
    \end{equation}
    The uniform global generation property of multiplier ideals next yields the existence of $m_0\in\N$ such that, for any $m\in\N$, the sheaf
    $$
    \cO_\cX((m+m_0)\cL)\otimes\fb_m(-m_0 D)
    $$
    is globally generated. This implies $\fb_m(-m_0 D)\subset\fa_{m+m_0}$, and hence 
    $$
    0\le\log|\fb_m|-\log|\fa_m|\le 2m_0 \phi_D
    $$
    for $m\ge m_0$ by the elementary Lemma~\ref{lem:comp} below, since $(\fa_m)$ is a graded sequence while $(\fb_m)$ is antigraded by the subadditivity property of multiplier ideals. Thus 
    $$
    0\le\int(\log|\fb_m|-\log|\fa_m|)\MA(\phi_m)\le 2m_0 \int\phi_D\MA(\phi_m),
    $$
    which combines with~\eqref{equ:bd1} to yield~\eqref{equ:bd}, and concludes the proof.
    \end{proof}

    \begin{lem}\label{lem:comp} Consider two sequences of vertical ideals $\fa_m,\fb_m\subset\cO_\cX$ on a model $\cX$ of $X$, and suppose that 
    \begin{itemize}
    \item[(i)] $(\fa_m)$ is a graded sequence, \ie $\fa_m\cdot\fa_p\subset\fa_{m+p}$ for all $m,p$;
    \item[(ii)] $(\fb_m)$ is antigraded, \ie $\fb_{m+p}\subset\fb_m\cdot\fb_p$ for all $m,p$;
    \item[(iii)] $\fa_m\subset\fb_m$ for all $m$, and $\fb_m(-F)\subset\fa_{m+m_0}$ for a fixed vertical divisor $F\ge 0$ and $m_0\in\N$. 
    \end{itemize}
    For any $m\geq m_0$ we then have
    $$
    0\le\log|\fb_m|-\log|\fa_m|\le 2\phi_F
    $$
    on $X^\an$. 
    \end{lem}

    \begin{proof} By (i), the sequence $m\mapsto\log|\fa_m|$ is superadditive, and hence 
    $$
    \log|\fa_\bullet|:=\lim_{m\to\infty} m^{-1} \log|\fa_m|=\sup_{m\ge 1} m^{-1}\log|\fa_m|
    $$
    pointwise on $X^\an$. By (ii), (iii) we further have $\fa_{m\ell}\subset\fb_{m\ell}\subset\fb_m^\ell$ for all $m,\ell$, and hence 
    \begin{equation}\label{equ:fb}
    m\log|\fa_\bullet|=\lim_\ell \ell^{-1}\log|\fa_{m\ell}|\le\log|\fb_m|,
    \end{equation}
    while (iii) also yields
    \begin{equation}\label{equ:fb1}
    \log|\fb_m|-\phi_F\le\log|\fa_{m+m_0}|. 
    \end{equation}
    For $m\ge m_0$ we infer 
    $$
    \log|\fb_m|-\phi_F\le (m+m_0)\log|\fa_\bullet|\le (m-m_0)\log|\fa_\bullet|\le\log|\fb_{m-m_0}|\le\log|\fa_m|+\phi_F
    $$
    where the last two inequalities follow from~\eqref{equ:fb},~\eqref{equ:fb1} with $m-m_0$ in place of $m$. The result follows. 
    \end{proof}

{\footnotesize
\bibliographystyle{acm}
\bibliography{main}
}
\end{document}